\documentclass[a4paper]{amsart}

\usepackage[utf8]{inputenc}
\usepackage[sortcites=true,style=numeric,url=false,isbn=false,giveninits=true,maxbibnames=9]{biblatex}
\renewbibmacro{in:}{\ifentrytype{article}{}{\printtext{\bibstring{in}\intitlepunct}}}
\usepackage{hyperref}
\usepackage{cleveref}

\newtheorem{theorem}{Theorem}[section]
\newtheorem{proposition}[theorem]{Proposition}
\newtheorem{lemma}[theorem]{Lemma}
\newtheorem{corollary}[theorem]{Corollary}

\theoremstyle{definition}
\newtheorem{definition}[theorem]{Definition}

\theoremstyle{remark}

\numberwithin{equation}{section}

\newcommand{\spt}{\operatorname{spt}}

\newcommand{\diam}{\operatorname{diam}}

\newcommand{\Tan}{\operatorname{Tan}}
\renewcommand{\d}{\, \mathrm{d}}
\newcommand{\muae}{$\mu$-a.e.\ }
\newcommand{\conv}{\operatorname{conv}}

\newcommand{\R}{\mathbb R}
\newcommand{\B}{C}
\newcommand{\U}{\mathcal U}
\newcommand{\tmu}{\tilde \mu}
\renewcommand{\L}{\mathcal L^1}
\renewcommand{\H}{\mathcal H^1}
\newcommand{\Ht}{\mathcal H^2}
\newcommand{\N}{\mathbb N}
\newcommand{\Z}{\mathbb Z}
\newcommand{\zo}{\{0,1\}^{\mathbb Z}}
\newcommand{\no}{\{0,1\}^{\mathbb N}}
\renewcommand{\S}{\mathcal S}
\newcommand{\T}{\mathcal T}
\renewcommand{\P}{\mathcal P}

\newcommand{\TT}{\mathbb R/2\mathbb Z}
\renewcommand{\th}{\Theta^1}
\newcommand{\dKR}{F}
\newcommand{\dGHs}{d_{*}}
\newcommand{\Mloc}{\mathcal M_{\mathrm{loc}}}
\newcommand{\M}{\mathbb M_*}

\begin{document}

\begin{abstract}
We construct a purely 2-unrectifiable metric space $X$ for which $\Ht(B(x,r))=(2r)^2$ for each $x\in X$ and $r>0$, providing a counter example to the Hausdorff density 1 conjecture for metric spaces.

The example arises as a product of two uniform metric measure space: a space $(Y,\nu)$, which satisfies $\nu(B(y,r))=cr$ for all $y\in Y$ and $r>0$.
We prove that there are exactly three 1-uniform metric measure spaces, $\R$ and two spaces we construct, $\S$ and $\T$.
The counterexample is the $\ell_1$-sum of $\S$ and $\R$.
\end{abstract}

\thanks{This work was supported by the Academy of Finland grant number 308510 and the European Union's Horizon 2020 research and innovation programme grant number 948021.
No generative AI was used in the preparation of this article.
}

\date{}
\author{David Bate}
	\email{david.bate@warwick.ac.uk}
\address{Zeeman Building, University of Warwick, Coventry CV4 7AL}

\title{A counterexample to the Hausdorff density conjecture for metric spaces}
\maketitle

\section{Introduction}

Kirchheim \cite{kirchheim-regularity} proved that, for any $n$-rectifiable metric space $X$ with $\mathcal H^n(X)<\infty$,
\begin{equation}
        \label{BPs}
\Theta^n(\mathcal H^n,x) := \lim_{r\to 0} \frac{\mathcal H^n(B(x,r))}{(2r)^n}
\end{equation}
equals 1 for $\mathcal H^n$-a.e.\ $x\in X$.
Since this work, it has been an open conjecture whether the converse holds \cite{lorent,delellis}.
To date, an answer is only known in dimension 1, for which the conjecture follows from the work Preiss--Ti{\v{s}}er \cite{preiss-tiser} on the Besicovitch 1/2-conjecture.
Other than this only partial results are known.
Lorent \cite{lorent} proves that any locally 2-uniform measure in $\ell_\infty^3$ is rectifiable.
Julia and Merlo \cite{MR4584667} which shows that subsets of homogeneous groups with a smooth box norm for which the Hausdorff density equals 1 must be low dimensional and are rectifiable.
The techniques of Marstrand \cite{zbMATH03177817} discussed below are combinatorial and the techniques can be made to apply in any metric space with a reasonable notion of ``strict convexity'' such as the existence of a \emph{unique} geodesic connecting any two points in the space.

We construct a counterexample to the conjecture in dimension 2.
Consequently, the conjecture is also false for all higher dimensions.
We now discuss the initial motivation of this work that led to our counterexample.

Geometric measure theory began with the pioneering work of Besicovitch.
One of the most striking, yet simplest to state results is the following \cite{besicovitchII}.
If $\mu$ is a Borel measure on the plane for which
\begin{equation}
        \label{BP}
\Theta^1(\mu,x) := \lim_{r\to 0} \frac{\mu(B(x,r))}{2r}
\end{equation}
exists and is positive and finite $\mu$-a.e, then $\mu$ is 1-rectifiable.
That is, $\mu$ almost all of $\R^2$ can be covered by countably many Lipschitz images of $\R$.
Besicovitch's argument was extended to 1-dimensional measures in higher dimension Euclidean spaces by \textcite{zbMATH03061001}, but question about higher dimensional measures, known as the \emph{Besicovitch problem} remained open for many years.
Partial results due to \textcite{zbMATH03177817} and \textcite{zbMATH03430988} proved the $n$-rectifiability of $\mu=\mathcal{H}^n|_E$ whenever
$\Theta^{n}(\mu,x)$
equals 1 almost everywhere.
Further, \textcite{zbMATH03232045} proved that if $\Theta^s(\mu,x)$ exists and is positive and finite $\mu$-a.e, then $s$ must be an integer.
It was not until 50 years after Besicovitch's work that \textcite{preiss} positively resolved the Besicovitch problem.

In this paper we consider the Besicovitch problem in arbitrary metric spaces.
To set the stage, we recall another result of \textcite{zbMATH03232045} which states that, if $\Theta^s(\mu,x)$ is positive and finite in any metric space, then $s\geq 1$.
Further, for $s>1$, if $\R$ is considered with the metric $d(x,y)=\sqrt[s]{|x-y|}$, then 1-dimensional Lebesgue measure satisfies $\mathcal L^1(B(x,r))=2r^s$ for every $x\in \R$ and $r>0$.
However, this metric space is purely 1-unrectifiable.
That is, $\H(\gamma)=0$ for every 1-rectifiable $\gamma\subset\sqrt[s]{\R}$.
In particular, $\sqrt[s]{\R}$ is not $n$-rectifiable for any $n\in\N$.
Thus, for general metric spaces, the remaining question is the validity of the 1-dimensional Besicovitch problem.

The argument of Besicovitch is of a combinatorial nature and in fact proves the rectifiability of measures satisfying \eqref{BP} in any uniformly convex Banach space.
Similarly to Marstrand's techniques, the argument applies to any metric space with a reasonable notion of ``strict convexity'' of balls.
However, the argument does not give any insight to the problem in $(\R^2,\|\cdot\|_\infty)$.
A related question is the Besicovitch $\tfrac{1}{2}$ conjecture \cite{preiss-tiser,2404.17536}.
However, even a solution of the most general form of this conjecture would not imply the rectifiability of measures satisfying \eqref{BP}.

One could reasonably expect that the 1-dimensional Besicovitch problem holds in any metric space.
The fact that any connected compact set of finite $\H$ measure is 1-rectifiable often allows for different and significantly simpler arguments for the case of 1-dimensional sets.
This is particularly evident in the proof of the Besicovitch problem in Euclidean spaces.
However, the expectation is in fact false.
\begin{theorem}\label{thm:example-uniform-pu}
For $i\in\Z$, let $e_i$ be the standard basis of $\ell_1$.
The set
\begin{equation*}
        \S:= \left\{\sum_{i\in\Z} 2^{i}s_i e_i: s_i \in \{0,1\},\ s_i =0 \text{ for sufficiently large } i\right\} \subset \ell_1
\end{equation*}
equipped with $\H$ is purely 1-unrectifiable and satisfies
\begin{equation}
\label{uniform}
\H(B(x,r))=r\quad \forall x\in\S,\ r>0.
\end{equation}
A measure satisfying \eqref{uniform} is called \emph{1-uniform}.
\end{theorem}

This space forms our counterexample.
\begin{theorem}
        \label{densityone}
        Let $\P$ be the $\ell_1$-sum of $\S$ and $\R$.
        Then for any $x\in \P$ and $r>0$,
        \[\Ht(B(x,r)) = (2r)^2.\]
\end{theorem}

Naturally, one now wonders if there are any other 1-uniform metric measure spaces.
Using $\S$ one can construct another 1-uniform space by placing a circle $\TT$ at those points in $\S$ with $s_i =0$ for all $i<0$ (see \Cref{def:torous}).
We denote this space by $\T$.
These two spaces, together with $\R$, are the only 1-uniform metric measure spaces.
\begin{theorem}\label{thm:main-uniform}
A 1-uniform metric measure space is isometric to a scaled copy of $\R$, $\S$ or $\T$.
\end{theorem}
It is interesting to observe that all of these spaces are equipped with a constant multiple of Hausdorff measure.
Also, any two points can be exchanged by an isometry of the space.

Measures that satisfy \eqref{BP} are called \emph{1-regular}.
With \Cref{thm:main-uniform} in hand, we prove a structure theorem for 1-regular metric measure spaces in terms of their tangent spaces.
Tangent \emph{measures} were introduced by Preiss in his solution of the Besicovitch problem in Euclidean space.
Roughly speaking, a tangent measure is a weak* limit of a Borel measure under dilations of the form $y\mapsto (y-x)/r_i$ as $r_i\to 0$, normalised by some scaling factor.
The set of all tangent measures to a measure $\mu$ at $x$ is denoted by $\Tan(\mu,x)$.
Tangent measures naturally generalise to the concept of a tangent \emph{space} of a metric measure space.
\begin{theorem}\label{regular_tan}
        Let $(X,d,\mu)$ be a 1-regular metric measure space.
        There exists a decomposition $X=E\cup S$ such that for $\mu$-a.e.\ $x\in E$, $\Tan(\mu,x)=\{\R\}$ and, for $\mu$-a.e.\ $x\in S$, $\Tan(\mu,x)=\{\lambda\mathcal S: \lambda>0\}$.
        Moreover, $\mu|_E$ is 1-rectifiable and $S$ is purely 1-unrectifiable.
\end{theorem}
The moreover part of \Cref{regular_tan} is given by work of the author \cite{mm-tan} which proves a characterisation of $n$-rectifiable metric spaces in terms of flat tangent spaces.
The rectifiability statement also follows from joint work with Valentine \cite{con-tan}, which proves the 1-rectifiability of metric measures spaces with connected tangents.
Tangent spaces of 1-regular spaces are 1-uniform.
Therefore, to prove \Cref{regular_tan} we must show that $\Tan(\mu,x)$ consists of \emph{either} $\R$ or $\S$, and not a combination of the two and $\T$.

\Cref{regular_tan} allows us to solve the 1-dimensional Besicovitch problem in various classes of metric spaces.
We refer the reader to \Cref{sec:BP} for the relevant definitions.
\begin{theorem}
        \label{solve_BP}
        Let $X$ be one of the following:
        \begin{enumerate}
                \item A doubling geodesic metric space.
                \item A uniformly convex Banach space.
                \item The Heisenberg group equipped with any homogeneous norm.
        \end{enumerate}
        Then any 1-regular measure on $X$ is 1-rectifiable.
\end{theorem}
\Cref{solve_BP} will follow from \Cref{regular_tan} by showing that $\S$ does not isometrically embed into any such $X$.

The article is structured as follows.
In \Cref{prelims} we give various preliminary lemmas that we will use throughout the argument.
We also recall the definition of a tangent space and several results related to them.
In \Cref{sec:S} we show that $\S$ and $\T$ are purely 1-unrectifiable and 1-uniform.
In \Cref{sec:densityone} we prove Theorem \ref{densityone}.
In \Cref{sec:besicovitch-pairs} we adapt the notion of a Besicovitch \emph{circle pair} from \cite{besicovitchII} so that the definition is suitable for working in an arbitrary metric space.
We then deduce several structural consequences of the existence of a Besicovitch pair in a 1-uniform metric measure space.
Since we consider metric spaces without any notion of strict convexity, our arguments completely diverge from those of Besicovitch.
\Cref{sec:classification} contains our classification of 1-uniform spaces using the results on Besicovitch pairs.
Finally, \Cref{sec:tangents} contains the proof of \Cref{regular_tan} and \Cref{sec:BP} the proof of \Cref{solve_BP}.

For notational simplicity, from this point on we will not use the ``1-'' prefix and simply refer to a rectifiable, purely unrectifiable, regular or uniform metric measure space.

\section{Preliminaries}
\label{prelims}
\subsection{Geometric properties of metric spaces}

Throughout this article, $(X,d)$ will denote a complete metric space and a \emph{measure} will refer to a $\sigma$-finite Borel regular (outer) measure.
For $\lambda >0$, let $\lambda X$ denote the metric space $(X,\lambda d)$.
For $x\in X$ and $r\geq 0$, let $B(x,r)$ be the \emph{open} ball of radius $r$ and $\B(x,r)$ the \emph{closure} of $B(x,r)$.

\begin{definition}\label{def:regular}
For a measure $\mu$ on $X$ we define the \emph{density} of $\mu$ to be
\begin{equation*}
        \th(\mu,x) := \lim_{r\to 0} \frac{\mu(B(x,r))}{2r}
\end{equation*}
whenever the limit exists.
Note that, in the case that the limit exists, we can use either open or closed balls or $\B(x,r)$ and obtain an equivalent definition and value of $\th(\mu,x)$.
The measure $\mu$ is \emph{regular} if $\th(\mu,x)$ exists $\mu$-a.e.
Further, $\mu$ is \emph{uniform} if $\mu(B(x,r))= r$ for every $x\in X$ and $r \geq 0$.
\end{definition}

The (1-dimensional) Hausdorff measure is defined by
\begin{equation*}
        \H(A) := \lim_{\delta \to 0} \inf\left\{\sum_{i\in \N} \diam A_i : A \subset \bigcup_{i\in\N} A_i\right\}
\end{equation*}
for any Borel subset $A$ of a metric space $X$.
Recall that there is a natural relationship between Hausdorff measure and density: if $\th(\mu,x) \leq \lambda$ for all $x$ is a Borel set $A$ then $\mu(A) \leq 2\lambda \H(A)$; if $\th(\mu,x) \geq \lambda$ for all $x\in A$ then $\mu(A) \geq \lambda \H(A)$.
See \cite[Section 2.10.19]{federer}.

\begin{definition}\label{def:rectifiable}
        A subset $S$ of a complete metric space $X$ is \emph{rectifiable} if there exist countably many $A_i \subset \R$ and Lipschitz $\gamma_i \colon A_i \to X$ such that
        \begin{equation*}
        \H\left(S\setminus \bigcup_i \gamma_i(A_i)\right)=0.
        \end{equation*}
\end{definition}
By a result of Kirchheim \cite[Lemma 4]{kirchheim-regularity}, for any $\epsilon>0$ an equivalent definition is obtained if each $\gamma_i$ is $(1+\epsilon)$-biLipschitz.
If $X$ is a Banach space, by linearly extending each $\gamma_i$ on the complement of each $A_i$, we get an equivalent definition if we require each $A_i=\R$.

The following result of Kirchheim \cite[Theorem 9]{kirchheim-regularity} shows that rectifiable sets are regular.
\begin{theorem}\label{thm:kirchheim}
        Let $X$ be a rectifiable metric space with $\H(X)<\infty$.
        Then \[\th(\H,x)=1\] for $\H$-a.e.\ $x\in X$.
\end{theorem}

Regular measures on a metric space are related in the following way.
\begin{lemma}\label{lem:regular-measures-are-equivalent}
Let $X$ be a metric space and $\mu,\nu$ two measures on $X$ such that
\begin{equation*}
        0< \th(\nu,x) \quad \text{and} \quad \th(\mu,x)< \infty
\end{equation*}
for every $x$ in some Borel $B\subset X$.
Then $\mu|_B \ll \nu|_B$ and the Radon-Nikodym derivative is given by
\begin{equation*}
        \frac{d\mu}{d\nu} = \frac{\th(\mu,)}{\th(\nu,)}
\end{equation*}
\end{lemma}

\begin{proof}
Since both $\mu|_B \ll \H|_B$ and $\H|_B \ll \nu|_B$, $\mu|_B \ll \nu|_B$.
To prove the statement about the Radon-Nikodym derivatives, it suffices to suppose that both $\mu$ and $\nu$ are finite.

To this end, first note that $\nu$ is asymptotically doubling.
That is,
\begin{equation*}
        \frac{\nu(B(x,2r))}{\nu(B(x,r))} < \infty 
\end{equation*}
for $\nu$-a.e.\ $x\in X$ and all sufficiently small $r>0$.
Therefore the Lebesgue differentiation theorem holds for $\nu$ integrable functions on $X$ \cite[Section 3.4]{zbMATH06397370}, such as $d\mu/d\nu$.
By applying it we obtain
\begin{equation*}
        \frac{d\mu}{d\nu}(x) = \lim_{r\to 0} \frac{1}{\nu(B(x,r))} \int_{B(x,r)} \frac{d\mu}{d\nu} d\nu = \lim_{r\to 0} \frac{\mu(B(x,r))}{\nu(B(x,r))} = \frac{\th(\mu,x)}{\th(\nu,x)}
\end{equation*}
for $\nu$-a.e.\ $x\in X$, as required.
\end{proof}

\begin{corollary}\label{cor:rectifiable-regular}
        Let $(X,d,\mu)$ be a metric measure space with $\H(X)<\infty$ for which $\th(\mu,x) = 1$ for $\H$-a.e.\ $x\in X$.
        Suppose that $R \subset X$ is rectifiable and has finite $\mu$ measure.
        Then $\mu|_R = \H|_R$.
\end{corollary}

\begin{proof}
        Since $\H(X)<\infty$, by \cite[Section 2.10.19]{federer}, $\th(\H|_{X\setminus R},x)=0$ for $\H$-a.e.\ $x\in R$.
        Since $\mu\ll \H$, $\th(\mu|_{X\setminus R},x)=0$ for $\H$-a.e.\ $x\in R$ and so $\th(\mu|_R,x)=1$ for $\H$-a.e.\ $x\in R$.
        By \Cref{thm:kirchheim}, $\th(\H|_R,x)=1$ for $\H$-a.e. $x\in R$, and so the result follows from applying \Cref{lem:regular-measures-are-equivalent} to the $\H$-full measure set where $\th(\H|_R,x)=\th(\mu|_R,x)=1$.
\end{proof}

The following result is standard.
\begin{lemma}\label{lem:uniform-is-compact}
        Any uniform metric measure space is \emph{proper}.
        That is, closed and bounded sets are compact.
\end{lemma}

\begin{proof}
        We show that any ball $B(x,r)$ can be covered by at most 5 closed balls of radius $r/2$.
        Inductively, this implies that any ball is totally bounded.
        Therefore, any closed and bounded set, being complete and totally bounded, is compact.

        Let $x_i$, $i\in I$ be a maximal $r/2$ separated subset of $B(x,r)$.
        Then $B(x,r) \subset \cup_{i\in I} B(x_i,r/2)$.
        Moreover the $B(x_i, r/4)$ are disjoint subsets of $B(x,5r/4)$ and so
        \begin{equation*}
        |I|r/4 = \sum_{i\in I} \tmu(B(x, r/4)) = \tmu(\cup_{i\in I} B(x, r/4)) \leq \tmu(B(x, 5r/4)) = 5r/4.
        \end{equation*}
        Therefore $|I|\leq 5$.
\end{proof}

We will inductively construct functions using the following lemma.
\begin{lemma}\label{lem:increasing-define-function}
        Let $(X,d), (Y,\rho)$ be two metric spaces, $S_1\subset S_2 \subset S_3 \subset \ldots$ an increasing sequence of subsets of $X$ and for each $n\in \N$, $f_n\colon S_n \to Y$ such that, for every $m \leq n$, $f_n|_{S_m}=f_m$. 
        If each $f_n$ is 1-Lipschitz/injective/an isometry, then so is
        \begin{align*}
                f &\colon \bigcup_{n\in \N} S_n \to Y\\
                f(x) &= f_n(x) \text{ whenever } x\in S_n.
        \end{align*}
\end{lemma}

\begin{proof}
If $x,y\in S$ then there exists an $n\in \N$ such that $x,y\in S_n$.
Then, if each $f_n$ is 1-Lipschitz,
\begin{equation*}
        \rho(f(x),f(y)) = \rho(f_n(x),f_n(y)) \leq d(x,y)
\end{equation*}
and hence $f$ is 1-Lipschitz.

If each $f_n$ is an isometry then the same argument with the inequality replaced by an equality shows that $f$ is an isometry.
A similar argument proves the result when each $f_n$ is injective.
\end{proof}

\subsection{Tangents of metric measure spaces}
\label{tangents}

\textcite{preiss} introduced the notion of a \emph{tangent measure} of a measure $\mu$ on Euclidean space as one step the solution of the classical Besicovitch problem.
In \cite{mm-tan}, these ideas were extended to define \emph{tangent spaces} of a metric measure space.
We begin by recalling several results and constructions from \cite{mm-tan}.

\begin{definition}[Definition 2.12 \cite{mm-tan}]
    \label{dKR}
    Let $X$ be a metric space and $x\in X$.
    For $\mu,\nu\in \Mloc(X)$, and $L,r>0$, define
    \[\dKR^{L,r}_x(\mu,\nu) = \sup\left\{\int g \d(\mu-\nu) : g\colon X\to [-1,1],\ L\text{-Lipschitz},\ \spt g\subset B(x,r)\right\}.\]
    Define $\dKR_x(\mu,\nu)$ to be the infimum, over all $0< \epsilon <1/2$, for which
    \[\dKR_x^{1/\epsilon,1/\epsilon}(\mu,\nu)<\epsilon.\]
    If no such $0<\epsilon<1/2$ exists, set $\dKR_x(\mu,\nu)=1/2$.
\end{definition}

\begin{definition}[Definition 4.9 \cite{mm-tan}]
	\label{dGHs}
        A \emph{pointed} metric measure space $(X,d,\mu,x)$ consists of a metric measure space and a distinguished point $x\in X$.

	Let $(X,d,\mu,x)$ and $(Y,\rho,\nu,y)$ be pointed metric measure spaces.
	Define \[\dGHs((X,d,\mu,x),(Y,\rho,\nu,y))\] to be the infimum, over all pointed metric spaces $(Z,z)$ and all isometric embeddings $(X,x),(Y,y)\to (Z,z)$ of $\dKR_z(\mu,\nu)$.
\end{definition}
Note that $\dKR_z^{L,r}$ is increasing in both $L$ and $r$.
By \cite[Corollary 4.13]{mm-tan}, $\dGHs$ is a complete and separable metric on $\M$, the equivalence classes of all pointed metric measure spaces under the relation
\[(X,\mu,x)\sim (Y,\nu,y) \text{ if } (\spt\mu\cup\{x\},\mu,x) \text{ is isometric to } (\spt\nu\cup\{y\},\nu,y).\]
Since an element of $\M$ is determined by the support of the measure alone, we will often omit the metric space from notation.
Similarly, we will only specify an explicit metric when necessary.

Given $(\mu,d,x)\in \M$ with $x\in\spt\mu$ and $r>0$, let
\[T_{r}(\mu,d,x) := \left(\frac{\mu}{\mu(B(x,r))},\frac{d}{r}, x\right)\]
\begin{definition}[Definition 5.1 \cite{mm-tan}]
	\label{def:gh-tangent}
	A $(\nu,\rho,y) \in \M$ is a \emph{tangent measure} of $(\mu,d,x)\in \M$ if there exist $r_k\to 0$ such that
	\[\dGHs(T_{r_k}(\mu,d,x),(\nu,\rho, y))\to 0.\]
	The set of all tangent measures of $(\mu,d,x)$ will be denoted by $\Tan(\mu,d,x)$.
\end{definition}

The tangent spaces of asymptotically doubling measures are compact metric spaces.
\begin{proposition}[Proposition 5.3 \cite{mm-tan}]
	\label{tangents_exist}
	Let $X$ be a metric space and let $\mu\in \Mloc(X)$ be asymptotically doubling.
	For \muae $x\in \spt\mu$ and every $r_0>0$,
	\begin{equation}\label{con_subseq}
	\{T_r(\mu,x): 0<r<r_0\} \text{ is pre-compact.}
		\end{equation}
	In particular, for any $x\in \spt\mu$ satisfying \eqref{con_subseq}, $\Tan(\mu,x)$ is a non-empty compact metric space when equipped with $\dGHs$ and
	\begin{equation}\label{contra}
	\forall \delta>0,\ \exists r_x>0 \text{ s.t. } \dGHs(T_{r}(\mu,x),\Tan(\mu,x)) \leq\delta\ \forall 0<r<r_x.
	\end{equation}
\end{proposition}

By the Lebesgue density theorem, tangents also pass to subsets.
\begin{lemma}[\cite{mm-tan} Corollary 5.7]
	\label{tangent-density-equal}
	Let $(\mu,d,x) \in \M$ and let $S\subset \spt\mu$ be $\mu$-measurable.
	Then for $\mu$-a.e. $x\in S$,
	\[\Tan(\mu\vert_S,x) = \Tan(\mu,x).\]
\end{lemma}

A near identical argument to \cite[Lemma 5.6]{mm-tan} gives the following.
\begin{lemma}
	\label{tan-of-unif}
	Let $(X,d,\mu)$ be a metric measure space and suppose, for $\mu$-a.e. $x\in X$,
	\[0<\th(\mu,x) \quad \text{and} \quad \th(\mu,x)<\infty.\]
	Then for $\mu$-a.e. $x\in X$, every $(\nu,y)\in\Tan(\mu,x)$ and every $z\in\spt \nu$,
	\[\frac{\th(\mu,x)}{\th(\mu,x)}
	\leq \frac{\nu(B(z,r))}{r}
	\leq \frac{\th(\mu,x)}{\th(\mu,x)}.\]
	In particular, if $\mu$ is regular, tangent spaces are regular.
\end{lemma}

Tangents of rectifiable sets are isometric to $\R$.
\begin{theorem}[Theorem 1.1 \cite{mm-tan}]\label{tan_pu}
	If $E\subset X$ is 1-rectifiable, then for $\H$-a.e. $x\in E$, $\Tan(E,\H,x) = \{(\R,\H/2)\}$.
\end{theorem}

In \cite[Theorem 1.1]{mm-tan} the converse statement is proven (in any dimension).
In fact, it suffices that the tangent spaces are supported on bi-Lipschitz images of Euclidean space.
For 1-dimensional metric spaces, it suffices that all tangent measures are supported on connected metric spaces \cite[Theorem 1.1]{con-tan}.
We record the following corollary.
\begin{theorem}\label{R_tan_rect}
	Suppose that for  $\H$-a.e.\ $x\in X$, $\th_*(\H,x)>0$ and \[\Tan(X,\H,x)=\{(\R,\H/2)\}.\]
        Then $X$ is 1-rectifiable.
\end{theorem}

We require an adaptation of Preiss's classical ``tangents of tangents are tangents''.
\begin{lemma}
	\label{tan-of-tan}
	For $(X,d,\mu,x)\in \M$ suppose that
        \[\Psi:=\lim_{\epsilon\to 0} \liminf_{r\to 0} \frac{\mu(B(x,r))}{\mu(B(x,r+\epsilon))}>0.\]
        For any $(Y,\rho,\nu,y)\in \Tan(X,d,\mu,x)$ and $s>0$, there exists $\Phi$ satisfying $\Psi\leq \Phi\leq 1$ such that
	\begin{equation*}
		\Phi T_s(\nu,\rho,y):=
		\left(\frac{\Phi\nu}{\nu(B(y,s))},\frac{\rho}{s},y\right) \in \Tan(X,d,\mu,x).
	\end{equation*}
	In particular, if $x$ satisfies \eqref{con_subseq},
	\begin{equation*}
		\Tan(\nu,\rho,y)\subset [1,\Psi^{-1}]\Tan(\mu,d,x).
	\end{equation*}
\end{lemma}

\begin{proof}
	Let $r_i\to 0$ such that $T_{r_i}(\mu,d,x)\to (\nu,\rho,y)$.
	Then
	\begin{equation*}
		\left( \frac{\mu(B(x,r_is))}{\nu(B(y,s)) \mu(B(x,r_i))}
		\frac{\mu}{\mu(B(x,r_is))},
		\frac{d}{r_is},x\right)
		\to T_s(\nu,\rho,y)
	\end{equation*}
	For any $\epsilon>0$,
	\begin{equation*}
	1\geq \frac{\mu(B(x,r_is))}{\nu(B(y,s)) \mu(B(x,r_i))}
		\geq \frac{\mu(B(x,r_i))}{\mu(B(x,r_i+\epsilon))}
	\end{equation*}
	for sufficiently large $i$.
	Therefore, by passing to a further subsequence, we may suppose that the inner quantity converges to some $\Phi$ satisfying the required condition.
	
	For any $s_i\to 0$, there exists $\Phi_i\in [1,\Psi^{-1}]$ such that $\Phi_i T_{s_i}(\nu,\rho,y) \in \Tan(\mu,d,x)$.
	Under the additional assumption, $\Tan(X,d,\mu,x)$ is compact and hence contains all limit points.
\end{proof}

\Cref{tangents_exist} shows that the set of tangent spaces is compact.
We conclude the preliminaries by showing that it is also connected.

\begin{lemma}
	\label{scale-cts}
	For any $(X,d,\mu,x)\in \M$, the map $\lambda \mapsto (X,\lambda d,\mu,x)$ is continuous.
\end{lemma}

\begin{proof}
	Let $\lambda,\lambda'>0$ and define the metric space $Z$ as the disjoint union of two copies of $X$, $Z=X \sqcup X'$, equipped with the metric $\widetilde\zeta$ defined to equal $\lambda d$ on $X$, $\lambda' d$ on $X'$ and
	\begin{equation*}
		\zeta(y,y') = \inf_{y\in X} \lambda d(y,w) + |\lambda-\lambda'| + \lambda' d(w,y')
	\end{equation*}
	whenever $y\in X$ and $y'\in X'$.
	(The triangle inequality for $\zeta$ is easily verified.)
	By applying \cite[Lemma 2.22]{mm-tan}, we obtain a metric space $(Z',\zeta')$ and isometric embeddings of $(X,\lambda d)$ and $(X',\lambda 'd)$ into $Z'$ such that $\zeta'\leq \zeta +\zeta(x,x')$ on $X\cup X'$ (for $x'$ the copy of $x$ in $X'$) mapping $x$ and $x'$ to a common point.
	In particular, for any $y\in X$, if $y'$ is its copy in $X'$ then $\zeta(y,y')\leq 2|\lambda-\lambda'|$.

	Now let $g\colon Z'\to [-1,1]$ be $L$-Lipschitz with $\spt g\subset B(x,r)$.
	Then
	\begin{equation*}
		|g(y)-g(y')|\leq L \zeta'(y,y')| \leq 2L|\lambda-\lambda'|
	\end{equation*}
	and so, if $\mu'$ denotes the copy of $\mu$ in $X'$,
	\begin{equation*}
		\dKR_z^{L,r}(\mu,\mu') \leq 2L|\lambda -\lambda'| \mu(B(x,r)).
	\end{equation*}
	Hence the map in the statement is continuous.
\end{proof}

\begin{corollary}\label{tan_is_con}
	Let $(X,d,\mu,x)\in\M$ be such that $\{T_r(X,d,\mu,x): 0<r<r_0 \}$ is pre-compact for some $r_0>0$.
	Then for any $\delta>0$ and $T_1,T_2 \in \Tan(\mu,x)$, there exists $0<r_1,r_2<r_0$ such that
	\begin{equation}
		\label{close_to_tan_delta}
		T_r(X,d,\mu,x),T_i)<\delta
	\end{equation}
	for $i=1,2$ and
        \begin{equation}
		\label{path-close-to-delta}
		T_r(X,d,\mu,x)\in B(\Tan(\mu,x),\delta)
        \end{equation}
	for each $r\in [r_1,r_2]$ (or $r\in [r_2,r_1]$ if non-empty).
        In particular, $\Tan(X,x)$ is connected.
\end{corollary}

\begin{proof}
	By \cref{tangents_exist}, there exists $R>0$ such that
	\begin{equation*}
		\dGHs(T_r(\mu,x),\Tan(\mu,x)) \leq \delta
	\end{equation*}
	for all $0<r<R$.
	Let $0<r_1,r_2<R$ satisfy \eqref{close_to_tan_delta}, so that \eqref{path-close-to-delta} is also satisfied.
	Therefore, by \cref{scale-cts}, there exists a continuous curve in $B(\Tan(\mu,x),\delta)$ connecting $T_1$ and $T_2$.

	Now suppose that $\Tan(\mu,x)= A_1\cup A_2$ is a disjoint decomposition into closed sets.
	By \cref{tangents_exist}, $\Tan(\mu,x)$ is compact, and hence so are $A_1$ and $A_2$.
	If both are non empty, let $T_i\in A_i$ and $\delta<d(A_1,A_2)/4$.
	Then the first part of the proof constructs a continuous curve from $T_1$ to $T_2$ that lies in $B(\Tan(\mu,x),\delta)$, a contradiction.
\end{proof}

\section{A purely unrectifiable uniform metric measure space}\label{sec:S}
We will consider the following subset of binary sequences with the following metric and measure.
\begin{definition}\label{def:tree}
        Let
        \begin{equation*}
                \S = \{s\in \zo : \exists \ n\in \N \text{ s.t. } s_i = 0 \ \forall \ i \geq n\}
        \end{equation*}
        equipped with the metric
\begin{equation*}
        d(s,t) = \sum_{i\in \Z} 2^{i}|s_i-t_i|
\end{equation*}
and 1-dimensional Hausdorff measure $\H$.
\end{definition}
Note that the metric is well defined on $\S$ and that $\S$ is isometric to the set in \Cref{thm:example-uniform-pu} in the introduction.
All distinguished points in $\S$ are equivalent: for any $x,y\in \S$ there exists an isometry of $\S$ exchanging $x$ and $y$.

We first show that $\S$ is uniform.
\begin{lemma}
        For any $x\in \S$ and $r\geq 0$, $\H(B(x,r)) = r$.
\end{lemma}

\begin{proof}
        For any $s\in \S$, the map
        \begin{align*}
                T \colon \S &\to \S\\
                t_i &\mapsto t_i + s_i \mod 2
        \end{align*}
        is an isometry that maps $s$ to $0$.
        Since $\H$ is preserved by isometries, it suffices to prove that $\H(B(0,r)) = r$ for each $r>0$.

        To see this, for any $r>0$ we may identify $r$ with its binary representation $r \in \S$, chosen such that $r$ does not have a recurring 1, and let $n\in \Z$ be the largest integer with $r_{n} \neq 0$.
        For each $j \leq n$ let
        \begin{equation*}
                r^j = \sum_{i=j}^n 2^i r_i.
        \end{equation*}
        Also let
        \begin{equation*}
                C = \{t \in \S : \exists m \text{ s.t. } t_i=1 \ \forall i \leq m\}.
        \end{equation*}
        Then $C$ is countable and so $\H(C)=0$.

        Fix $m \leq n$.
        We will estimate the Hausdorff measure of $B(0,r)\setminus C$ from above by covering it by sets of diameter at most $2^{m}$.
        First, for any
        \begin{equation*}p=(p_{m},p_{m+1},\ldots, p_{n-2}, p_{n-1}) \in \{0,1\}^{n-m}\end{equation*}
        let
        \begin{equation*}
                S(p) = \{t\in B(0,r) : t_{i} = p_{i} \ \forall \ m \leq i < n\}.
        \end{equation*}
        Then for any $s,t\in S(p)$,
        \begin{equation*}
                d(s,t) \leq \sum_{i<m} 2^i |s_i-t_i| \leq 2^m,
        \end{equation*}
        so that each $S(p)$ has diameter at most $2^m$.
        Moreover, if $d(0,s)<r^n$, then $s_i=0$ for each $i \geq n$, so that
        \begin{equation*}
                B(0,r^n) \subset \bigcup \{S(p) : p\in \{0,1\}^{n-m}\}.
        \end{equation*}

        Second, for any $i < n$ with $r_i =1$ let
        \begin{equation*}
                S_i = \{s\in B(0,r) : s_j = r_j \ \forall \ i < j \leq n, s_{i}=0\}.
        \end{equation*}
        If $s\not\in C$ and $d(0,s) \geq r^n$ then $s_n = r_n=1$.
        Also, since $r$ does not have a recurring 1, if $d(0,s) \leq r$ then $s_i \leq r_i$ for all $i\in \Z$, so that if $d(0,s)<r$ then $s_i <r_i$ for some $i<n$.
        Therefore, we may cover $(B(0,r)\setminus C)\setminus B(0,r^n)$ by $\cup \{S_i: i<n,\ r_i =1\}$.
        Finally, notice that the diameter of $S_i$ is at most $2^i$.

        By combining these observations, $B(0,r)$ is covered by
        \begin{equation*}
                \bigcup \{S(p) : p\in \{0,1\}^{n-m}\} \cup \bigcup \{S_i : i<n,\ r_i=1\}.
        \end{equation*}
        Since each of these sets have diameter at most $2^m$, with $m \leq n$ chosen arbitrarily, we see that
        \begin{align*}
                \H(B(0,r)) &\leq \sum_{p=1}^{2^{n-m}} 2^m + \sum_{\substack{i<n\\ r_i=1}} 2^i\\
                           &= 2^n + \sum_{i<n}2^i r_i = r.
        \end{align*}

        To see the other inequality, note that the map
        \begin{align*}
                B(0,r) &\to [0,r)\\
                s &\mapsto \sum_{i \in \Z} 2^i s_i
        \end{align*}
        is 1-Lipschitz and surjective, so that $\H(B(0,r)) \geq \H([0,r)) =r$.
        Therefore, $\H(B(0,r)) = r$.
\end{proof}

Next we show that $S$ is purely unrectifiable.

\begin{lemma}\label{lem:T-pu}
        Let $\gamma \colon A\subset \R \to S$ be biLipschitz.
        Then $\L(A)=0$.
\end{lemma}

\begin{proof}
        Let $\gamma$ be $L$-biLipschitz.
        For any $j\in \N$, we can cover $S$ by $2^j$ sets of the form $T_j(x_1),T_j(x_2),\ldots,T_j(x_{2^j})$ of diameter $2^{-j}$.
        These sets are separated by distance at least $2^{-j}$.
        Therefore, the sets $A_j^i=\gamma^{-1}(T_j(x_i))$ have diameter at most $L\cdot 2^{-j}$ and are separated by distance at least $2^{-j}/L$.
        Since $A\subset \cup_i A_j^i$ for each $j\in \N$, $A$ cannot have any density points and so must have measure zero.
\end{proof}
This lemma establishes \Cref{thm:example-uniform-pu}.

Next we consider the following space.
\begin{definition}\label{def:torous}
        Let
        \begin{equation*}
                \T = \{(x,s) \in \TT \times \no : \exists \ n \in \N \text{ s.t. } s_i =0 \ \forall i\geq n\}
        \end{equation*}
        equipped with the metric
        \begin{equation*}
                d((x,s),(y,t)) = d(x,y) + \sum_{i\in \N} 2^i |s_i-t_i|
        \end{equation*}
        and scaled 1-dimensional Hausdorff measure $\H/2$.
\end{definition}
As is the case for $\S$, all distinguished points in $\T$ are equivalent.

The set $\T$ is 1-rectifiable, since it is a countable union of circles $\TT$.
Also $\H$ restricted to $\TT$ is simply Lebesgue measure.
The space $\T$ is also uniform.

\begin{lemma}\label{lem:T-is-uniform}
        For any $x\in \T$ and $r\geq 0$, $\H(B(x,r))=2r$.
\end{lemma}

\begin{proof}
        Let $(x,s)\in \TT\times \no$ and $r>0$.
        Let $n\in \N$ be maximal with $n < r$ and let $r'=r-n$, so that $0<r' \leq 1$.
        Then, if $d((x,s),(y,t))<r$,
        \begin{equation*}
                d(x,y) + \sum_{i\in\N} 2^i |s_i-t_i| < r' + n,
        \end{equation*}
        so that
        \begin{equation}\label{eq:equality-inequality}
                \sum_{i\in\N} 2^i |s_i-t_i| \leq n,
        \end{equation}
        and, if \cref{eq:equality-inequality} is an equality, then $d(x,y) < r'$.
        Note that there is only one value of $t$, which we denote by $t^*$, that makes \cref{eq:equality-inequality} an equality, and it is such that $|s-t|$ is the binary representation of $n$.
        That is,
        \begin{equation}\label{eq:T-ball-decomp}
                B((x,s),r) \subset \bigcup\{\TT \times \{t\} : \sum_{i\in\N} 2^i |s_i-t_i| <n\} \cup B(x,r') \times\{t^*\}.
        \end{equation}
        A direct calculation shows that the reverse inclusion, and hence equality, holds.
        Also note that this decomposition is disjoint.
        
        Since the values of $t$ that satisfy \cref{eq:equality-inequality} as a strict inequality correspond to the binary representations of those integers $0 \leq m <n$,
        Since \cref{eq:T-ball-decomp} is a disjoint decomposition, we have
        \begin{align*}
                \H(B((x,s),r)) &= n \H(\TT) + B(x,r')\\
                               &= n 2 + 2r' = 2r.
        \end{align*}
\end{proof}

\begin{definition}
	\label{uniform-space}
	In what follows we write $\R$ for the real line equipped with the measure $\H/2$, so that $\R$ is uniform.

	Define
	\begin{equation*}
		\U =\{\R,T_r \T, T_r\S: r>0\}.
	\end{equation*}
\end{definition}

If $X$ is uniform, then so is $\T_r X$ for all $r>0$ and so the elements of $\U$ are all uniform spaces.
In Sections \ref{sec:besicovitch-pairs} and \ref{sec:classification}, we show that $\U$ contains all 1-uniform spaces.

\section{Hausdorff density conjecture}
\label{sec:densityone}

\begin{definition}
        Let $\P$ denote the $\ell_1$-sum of $\S$ and $\R$:
        \[d((s,x),(t,y))=d(s,t)+|x-y|.\]
        We consider $\P$ as a subset of $\ell_1\times \R$.
\end{definition}
Fubini's theorem naturally defines a uniform measure on the product of two uniform spaces.
Indeed, on $\P$ define the product measure $\mu=\H\times \H$.
Then, for any $(s,x)\in\P$ and $r>0$,
\begin{equation}\label{fubini}
        \mu(B((s,x),r))=\int_{0}^r 2(r-t)\d t= r^2.
\end{equation}
Thus we wish to show that $\mu = \Ht/4$.

In $\ell_1$ it is not possible to contain every set of diameter $\delta$ in a ball of radius $\delta/2$.
The key observation is that this is possible for subsets of $\S$ and $\P$.

\begin{lemma}
        \label{diameter-inductive}
        For $M\in\N$ let $\delta^1=\sum_{|i|\leq M} \delta^1_i$, $\delta^2=\sum_{|i|\leq M} \delta^2_i 2^i$ with $\delta^1_i\neq \delta^2_i \in \{0,1\}$ and let $\delta =\delta_1+\delta_2$.
        Suppose that $a,b\in \partial B(0,\delta/2)\subset \ell_1^2$ with $\|a-b\|=\delta^1$ and $c\in \ell_1^2$ with $\|a-c\|=\delta$ and $\|b-c\|=\delta^2$.
        Then $c\in \partial B(0,\delta)$.
\end{lemma}
\begin{proof}
        Up to an isometry we have $a=(-\delta^2/2,\delta^1/2)$, $b=(-\delta^2/2,-\delta^1/2)$ and $c=(\delta^2/2,-\delta^1/2)$.
\end{proof}

\begin{proposition}
        \label{points-in-ball}
        Let $F\subset \P$ be finite with $\diam F=\delta$.
        There exists a ball with $F\subset \partial B(c,\delta)$.
\end{proposition}

\begin{proof}
        In the proof, $\partial B$ refers to the boundary of a ball relative to a finite dimensional subspace that we work with.

        We prove the result by induction and first prove the case that $F\subset \S$.
        Let $s,t$ maximise the distances in $F$ and set $F'=F\setminus \{s\}$.
        If $|F'|=1$ then we are done, otherwise let $F'\subset \partial B(c',\diam F')$.
        Let $b\in\partial B(c',\diam F')\cap \conv F'\cap S$ be such that $b-c$ contains no components of $\conv F'-c'$ and set $\delta^1=\|b-a\|$.
        There exist two balls $B,B'$ in $\ell_1$ of radius $\diam F$ such that
        \[\partial B \cap \conv F'= \partial B(c',\delta^1) = \partial B' \cap \conv F'.\]
        By Lemma \ref{diameter-inductive}, $s$ lies on $\partial B$ or $\partial B'$.
        This completes the proof for the case $F\subset \S$.

        Now suppose $F\subset \P$.
        Let $(s,x),(t,y)$ attain $\diam F=:\delta$ and set $F'=F\setminus \{(t,y)\}$.
        By induction hypothesis there exists a ball $B(c,r)\supset F'$ with
        \[\{p\in\S: \exists x\in \R (p,x)\in F'\}\subset \partial B(c,r).\]
        By the $F\subset \S$ case, there exists a ball $\partial B(c',r')$ containing
        \[\S_y:=\{(p,y):p\in F,\ \exists z\in \R,\ (p,z)\in F\}\]
        and $(t,y)$.
        Note that the diameter of this ball, $2r'$, is at least $d(x,y)$.
        Similarly to the previous case, there exist two balls $B,B'$ of diameter $\delta$ with
        \[\partial B \cap \conv F'= \partial B(c,r) = \partial B' \cap \conv F'.\]
        Since $2r'\geq d(x,y)$, one of them satisfies $(t,y)\in \partial B$.
\end{proof}

Necessarily, the following inequality is sharp.
\begin{proposition}
        \label{measure}
        For any $S\subset \P$, $\mu(S)\leq \diam(S)^2/4$.
\end{proposition}

\begin{proof}
We may suppose $S$ is closed and bounded.
Since $\S$ and $\R$ are doubling, so is $\P$ and hence $S$ is compact.
Let $N_\epsilon$ be an increasing collection of $\epsilon$-nets in $S$.
Then the closed convex hull of $N_\epsilon$ intersected with $S$ monotonically increases to $S$.
Thus it suffices to prove the inequality for $S$ contained in such a convex set.

By Proposition \ref{points-in-ball}, $S$ is contained within a ball $B$ of diameter $\diam S$.
If the centre of this ball were in $\S$ we would be done by \eqref{fubini}.
Since this may not be the case, we approximate this calculation.
Fix $m\in\Z$ and let
\[P=\{p\in B : p_i=0\ \forall i<m\}.\]
For $p\in P$ let $V_p=B\cap B(p,2^m)\times \R$, which has diameter $2(\delta/2-\|p-c\|+2^{m+1})$.
Thus
\begin{align*}
        \mu(S) &\leq \sum_{p\in P} 2^m 2 (\delta/2 -\|p-c\|+2^{m+1}).
\end{align*}
As $p$ varies in $P$, $\delta/2 -\|p-c\|$ takes $k=\lceil(\delta/2) 2^{-m}\rceil$ distinct values ranging from $0$ to $\delta/2$ in increments of $2^m$.
Thus
\begin{align*}
        \mu(S) &\leq
\frac{k(k+1)}{2}2^m 2^m 2
+ 2\frac{k(k+1)}{2}2^m 2^{m+1}2^m
\\
&\leq (\delta/2)^2 + (\delta/2)2^m + 2 (\delta/2)^2 2^m \to (\delta/2)^2.
\end{align*}
\end{proof}

\begin{proof}[Proof of \Cref{densityone}]
        By \cite[Section 2.10.18]{federer},
        \[\limsup_{r\to 0}\frac{\Ht(B(x,r))}{(2r)^n} \leq 1\]
        for $\Ht$-a.e.\ $x\in \P$.
        By scale and translation invariance, $\Ht(x,r) \leq (2r)^2$ for all $x\in \P$ and $r>0$.
        On the other hand, by Proposition \ref{measure} and \eqref{fubini}, for all $x\in\P$ and $r>0$,
        \[4r^2 = 4\mu(B(x,r))\leq \Ht(B(x,r))\leq (2r)^2.\]
\end{proof}

\section{Besicovitch pairs}

\label{sec:besicovitch-pairs}
The work of \textcite{besicovitchII} relies on finding a certain configuration of points called a Besicovitch pair.
Besicovitch pairs also feature in the work of \textcite{preiss-tiser,2404.17536}.
We will also make heavy use of Besicovitch pairs, but we require a slightly different formulation than \cite{besicovitchII} (but the two are equivalent in Euclidean space) in order to work in an arbitrary metric space.

\begin{definition}\label{def:besicovitch-pair}
        A pair of points $a,b$ in a metric space is a \emph{Besicovitch pair} if for any $x\in B(a,d(a,b))$ and any $y\in B(b,d(a,b))$,
        \begin{equation*}\label{eq:besicovitch-pair}
                d(x,y) \geq d(a,b).
        \end{equation*}
\end{definition}
Of particular importance is the relationship between Besicovitch pairs and connectivity of the space.
In \cite{besicovitchII}, Besicovitch pairs are found as a consequence of a set of finite $\H$ measure being purely unrectifiable.
In the presence of a uniform measure, we will find them as a consequence of a ball being disconnected.

We emphasise that, although the definition of a Besicovitch pair is motivated by the approach in \cite{besicovitchII}, those techniques strongly rely on the strict convexity of the ball in $\R^2$, which is not true in general.
As a result, our method is entirely different.

First we prove several metric properties of a uniform metric measure space.
Recall that $\B(x,r)$ is the closure of the open ball $B(x,r)$.
\begin{lemma}\label{lem:uniform-properties}
        Let $(X,d,\mu)$ be a uniform metric measure space, $x\in X$ and $r>0$.
        \begin{enumerate}
                \item \label{item:point-outside-ball} If $\B(x,r)$ is open then there exists a $y\not\in \B(x,r)$ with $d(x,y)=r$.
                \item \label{item:closure-of-open} If $U \subset X$ is open and $F\subset U$ is a full measure subset. 
                        Then the closure of $F$ equals the closure of $U$.
                \item \label{item:uniform-disjoint-subballs} If $B(x,r)$ contains disjoint balls $B(y_1,r_1)$ and $B(y_2,r_2)$ with $r=r_1+r_2$ and $r_1,r_2\geq 0$ then
                        \begin{equation*}
                                \B(x,r) = \B(y_1,r_1)\cup \B(y_2,r_2).
                        \end{equation*}
                \item \label{item:pre-l1-distance} If $y,z\in X$ with $B(y,d(y,z)) \cap B(x,d(x,y))=\emptyset$ and $d(x,z)>d(x,y)$, then 
                        \begin{equation*}
                                d(x,z) = d(x,y) + d(y,z).
                        \end{equation*}
                \item \label{item:single-point-boundary} There can be at most one point $y\in X$ with $d(x,y)=r$ for which there exists a $\delta>0$ such that
                        \begin{equation*}
                                B(x,r) \cap B(y,\delta) = \emptyset.
                        \end{equation*}
        \end{enumerate}
\end{lemma}

\begin{proof}
        To prove \cref{item:point-outside-ball}, for each $\epsilon>0$
        \begin{equation*}
                \mu(B(x,r+\epsilon)) = r+\epsilon > r = \mu(\B(x,r))
        \end{equation*}
        and so there exists a $y \in B(x,r+\epsilon) \setminus \B(x,r)$.
        Since this is true for all $\epsilon>0$ and since $\B(b,r+1)$ is compact, there exists a $y\in \B(x,r+1)\setminus \B(b,r)$ with $d(x,y)=r$, as required.

        Since $U$ minus the closure of $F$ is open, 
        if it contains a point it must contain an open ball, which necessarily has positive measure.
        This is impossible since $F$ is a full measure subset of $U$.
        Therefore, the closure of $U$ is contained in the closure of $F$.
        Since $F\subset U$, this proves \cref{item:closure-of-open}.

        To prove \cref{item:uniform-disjoint-subballs}, observe that, since $\mu$ is uniform and $B(y_1,r_1)$ and $B(y_2,r_2)$ are disjoint subsets of $B(x,r)$,
        \begin{equation*}
                \mu(B(x,r) \setminus (B(x,r_1) \cup B(y,r_2))) = r- r_1 - r_2 = 0.
        \end{equation*}
        Thus, \cref{item:closure-of-open} gives the result.

        To prove \cref{item:pre-l1-distance}, observe that for any $0<r<d(y,z)$, $B(y,r)$ and $B(x,d(x,y))$ are disjoint subsets of $B(x,d(x,y)+r)$, and so
        \begin{align*}
                B(x,d(x,y) + d(y,z)) &= \bigcup_{0<r<d(y,z)} \B(x,d(x,y) +r)\\
                                     &= \bigcup_{0<r<d(y,z)} \B(x,d(x,y)) \cup \B(y,r)\\
                                     &= \B(x,d(x,y)) \cup B(y,d(y,z))
        \end{align*}
        using \cref{item:uniform-disjoint-subballs} for the second equality.
        Therefore, since $d(x,z) > d(x,y)$, $z \not\in B(x,d(x,y)+d(y,z))$
        and so $d(x,z) \geq d(x,y)+d(y,z)$.
        Therefore, by the triangle inequality, $d(x,z) = d(x,y)+d(y,z)$.

        Finally, suppose for a contradiction to \cref{item:single-point-boundary}, that there exist two points $y_1 \neq y_2$ and $\delta_1,\delta_2>0$ such that
        \begin{equation*}
                B(x,r) \cap B(y_i,\delta_i) = \emptyset
        \end{equation*}
        for $i=1,2$ and set $\delta= \min(\delta_1,\delta_2, d(y_1,y_2)/2)$.
        Then $B(y_1,\delta),B(y_2,\delta)$ and $B(x,r)$ are disjoint subsets of $B(x,r+\delta)$ and so
        \begin{equation*}
                2\delta +r = \mu(B(y_1,\delta)\cup B(y_2,\delta) \cup B(x,r)) \leq \mu(B(x,r+\delta)) = r+\delta,
        \end{equation*}
        a contradiction.
\end{proof}

First we establish the existence of Besicovitch pairs in disconnected balls.
\begin{proposition}\label{prop:existence-of-pairs}
        Let $(X,d)$ be a proper metric space that satisfies \Cref{lem:uniform-properties} \cref{item:uniform-disjoint-subballs} and suppose that, for some $x\in X$ and $r>0$, $\B(x,r)$ is disconnected.
        Then there exists a Besicovitch pair $a,b$ with $a \in \B(x,r)$ and $d(a,b) \leq r$.
\end{proposition}

\begin{proof}
        Let $\B(x,r)=A\cup B$ be a decomposition into disjoint closed sets with $x\in A$ and $B \neq \emptyset$.
        Since $B$ is a closed subset of $\B(x,r)$ (which is compact since $X$ is proper), there exists a closest point $y$ of $B$ to $x$.
        In particular, $B(x,d(x,y))\subset A$ and, since $A$ is closed, $d(y,A)>0$.
        Therefore,
        \begin{equation*}
                B(y,d(y,A)/2) \cap B(x,d(x,y)) \subset B(y,d(y,A)/2) \cap A =\emptyset.
        \end{equation*}
        By \Cref{lem:uniform-properties} \cref{item:uniform-disjoint-subballs},
        \begin{multline*}
                \B(x,d(x,y)+ d(y,A)/2) = \B(x,d(x,y)) \cup \B(y,d(y,A)/2)\\
                =(\B(x,d(x,y))\cap A) \cup (\B(x,d(x,y))\cap B) \cup \B(y,d(y,A)/2)
        \end{multline*}
        and so $Y:= \B(x,d(x,y))\cap A$ is open since it contains a ball of radius
        \begin{equation*}
                \min(d(y,A)/2, d(A,B))>0
        \end{equation*}
        around each of its points.
        Since $A$ is a closed subset of $\B(x,r)$, $Y$ is also compact.

        To conclude, since $Y$ is compact with closed complement $Y^c$, we may find $a\in Y$ and $b\in Y^c$ with $d(a,b)=d(Y,Y^c)$.
        This implies that $a,b$ is a Besicovitch pair.
        Indeed, if $x\in B(a,d(a,b))$ and $y\in B(b,d(a,b))$ then necessarily we have $x\in Y$ and $b\in Y^c$ and hence $d(x,y)\geq d(Y,Y^c) = d(a,b)$.
        By construction we have $a\in Y \subset \B(x,r)$ and $d(a,b) \leq d(Y,Y^c) \leq d(x,y) \leq r$, as required.
\end{proof}

\Cref{lem:uniform-properties} \cref{item:single-point-boundary} allows us to define the following mapping between the components of a Besicovitch pair.
\begin{definition}\label{def:neighbour}
        Let $(X,d,\mu)$ be a uniform metric measure space and $a,b \in X$ a Besicovitch pair with $d=d(a,b)$.
        Define
        \begin{equation*}
                \iota_{ab} \colon \B(a,d)\cup \B(b,d) \to \B(a,d) \cup \B(b,d)
        \end{equation*}
        by mapping any point in $\B(a,d)$ to its closest point in $\B(b,d)$ and vice versa.
\end{definition}
Indeed, $\iota_{ab}$ is well defined since \emph{any} closest point $y \in \B(b,d)$ to $x\in \B(a,d)$ satisfies $B(y,d) \cap \B(a,d)=\emptyset$, by the definition of a Besicovitch pair.

We now establish various properties of Besicovitch pairs in uniform metric measure spaces.
\begin{lemma}\label{lem:bpair-properties}
        Let $a,b$ be a Besicovitch pair in a uniform metric measure space $X$ and $d=d(a,b)$.
        Then
        \begin{enumerate}
                \item \label{item:bc-double-ball} $\B(a,2d) = \B(a,d) \cup \B(b,d)$.
                \item \label{item:bc-components-are-open} For any $x\in \B(a,d)$, $\B(x,d)=\B(a,d)$ and $\B(a,d)$ is open.
                \item \label{item:bc-fixed-distance} $d(x,\iota_{ab}(x)) = d$ for each $x\in \B(a,d) \cup \B(b,d)$.
                \item \label{item:bc-l1-dist} For any $x\in \B(a,d)$ and $z\in \B(b,d)$, $d(x,z) = d + d(\iota_{ab}(x), z)$.
                \item \label{item:bc-self-inverse} $\iota_{ab}^2 = \iota_{ab}$.
                \item \label{item:bc-balls-are-separated} $\B(x,2d) = \B(a,d)\cup \B(b,d)$ for any $x\in \B(a,d) \cup \B(b,d)$.
        \end{enumerate}
\end{lemma}

\begin{proof}
        Applying \Cref{lem:uniform-properties} \cref{item:uniform-disjoint-subballs} with $r_1=r_2=d$ gives \cref{item:bc-double-ball}.

        To prove \cref{item:bc-components-are-open}, let $x\in \B(a,d)$.
        By the triangle inequality, $\B(x,d) \subset \B(a,2d)$, but by the definition of a Besicovitch pair, $d(x,z)\geq d$ for any $z\in \B(b,d)$.
        Therefore, by \cref{item:bc-double-ball}, $\B(x,d) \subset \B(a,d)$.
        By \Cref{lem:uniform-properties} \cref{item:uniform-disjoint-subballs}, $\B(x,d)=\B(a,d)$.
        Therefore, by \cref{item:bc-double-ball}, $\B(x,d)$ contains an open ball of radius $d$ around each of its points and so is open.

        By \Cref{lem:uniform-properties} \cref{item:point-outside-ball} there exists a $y\not\in \B(x,d)$ with $d(x,y)=d$.
        By \cref{item:bc-double-ball}, we must have $y\in\B(b,d)$.
        Since $a,b$ is a Besicovitch pair, $d(x,\B(b,d)=d(x,y)$ and so $\iota_{ab}(x)=y$, proving \cref{item:bc-fixed-distance}.
        By symmetry, the corresponding statement is true for $x\in\B(b,d)$.

        For \cref{item:bc-l1-dist}, let $x\in \B(a,d)$ and $z\in \B(b,d)$.
        Necessarily, we must have $d \leq d(x,z) \leq 2d$.
        If $d=d(x,z)$, then since $\B(b,d)$ is open, by \Cref{lem:bpair-properties} \cref{item:single-point-boundary}, $z=\iota_{ab}(x)$, and so \cref{item:bc-l1-dist} holds.
        If $d< d(x,z)<2d$, then by applying \Cref{lem:uniform-properties} \cref{item:pre-l1-distance} with $y=\iota_{ab}(x)$, $d(x,z) = d + d(\iota_{ab}(x),z)$.
        If $d(x,z)=2d$, since $z$ is a limit of a sequence in $\B(b,d)$, for which the equality holds, continuity of the distance gives \cref{item:bc-l1-dist}.

        If $x\in \B(a,b)$ then $\iota_{ab}(x) \in \B(b,d)$ with $d(x,\iota_{ab}(x))=d$.
        Since $\iota_{ab}(\iota_{ab}(x))$ is the unique point in $\B(a,d)$ with $d(\iota_{ab}(\iota_{ab}(x)),\iota_{ab}(x))=d$, we must have $\iota_{ab}(\iota_{ab}(x))=x$, proving \cref{item:bc-self-inverse}.

        Finally, if $x\in \B(a,d)$, by \cref{item:bc-l1-dist} and the triangle inequality,
        \begin{equation*}
                \B(x,2d) \supset \B(a,d)\cup \B(\iota_{ab}(x),d).
        \end{equation*}
        By \Cref{lem:uniform-properties} \cref{item:uniform-disjoint-subballs} we have equality.
\end{proof}

\begin{proposition}\label{prop:iota-isometry}
        Let $(X,d,\mu)$ be a uniform metric measure space and $a,b\in X$ a Besicovitch pair.
        Then $\iota_{ab}$ is an isometry.
\end{proposition}

\begin{proof}
        Let $d=d(a,b)$ and write $\iota = \iota_{ab}$.
        Let $x\in \B(x,d)$ and $z\in \B(b,d)$.
        By \Cref{lem:bpair-properties} \cref{item:bc-l1-dist} and \cref{item:bc-self-inverse},
        \begin{equation*}
                d(\iota(x),\iota(z)) = d + d(\iota(\iota(x)), \iota(z)) = d + d(x, \iota(z)) = d(x,z).
        \end{equation*}
        Therefore, $\iota$ is an isometry for this configuration.

        For the other configuration, by symmetry we may suppose that $x,y \in \B(a,b)$.
        In this case, again by \cref{item:bc-l1-dist} and \cref{item:bc-self-inverse}, we have
        \begin{equation*}
                d(x,\iota(y)) = d + d(\iota(x),\iota(y))
        \end{equation*}
        and
        \begin{equation*}
                d(\iota(y),x) = d + d(\iota(\iota(y)),x) = d+ d(y,x).
        \end{equation*}
        Therefore, $d(x,y)=d(\iota(x),\iota(y))$, as required.
\end{proof}

\begin{proposition}\label{prop:double-pair}
        Let $(X,d,\mu)$ be a uniform metric measure space and $a,b\in X$ a Besicovitch pair.
        There exists a $b'\in X$ with $d(a,b') = 2d(a,b)$ such that $a,b'$ is a Besicovitch pair.
\end{proposition}

\begin{proof}
        Let $d=d(a,b)$.
        By \Cref{lem:bpair-properties} \cref{item:bc-double-ball,item:bc-components-are-open}, $\B(a,2d)$ is open.
        Therefore, by \Cref{lem:uniform-properties} \cref{item:point-outside-ball}, there exists a $b'\not\in \B(a,2d)$ with $d(a,b')=2d$.

        By \Cref{lem:bpair-properties}, for any $x\in \B(a,2d)$, $d(x,b) \geq 2d$ and so $B(a,2d) \cap B(b',2d) = \emptyset$, as required.
\end{proof}

The previous results imply a structure theorem for uniform spaces with a Besicovitch pair.
\begin{theorem}\label{thm:bp-isometry-to-l1}
        Let $(X,d,\mu)$ be a uniform metric measure space and $a,b\in X$ a Besicovitch pair with $d=d(a,b)$.
        Then $X$ is isometric to
        \begin{equation*}
                \B(a,d) \times \{0,1\}^\N
        \end{equation*}
        with the metric
        \begin{equation*}
                d((x,s),(y,t)) = d(x,y) + d \sum_{i\in \N} 2^i |s_i-t_i|
        \end{equation*}
\end{theorem}

\begin{proof}
        For each $n \in \N$ we will inductively construct a surjective isometry
        \begin{equation*}
                \iota_n \colon \B(a,2^n d)\to \B(a,d)\times \{0,1\}^{n-1} \times \{0\}^\N \subset \B(a,d) \times \{0,1\}^\N
        \end{equation*}
        such that, for any $m \leq n$, $\iota_n$ and $\iota_m$ agree on their common domain.

        Once we have done this, the proof is completed by defining $\iota$, an isometry between $X$ and $\B(a,d)\times\{0,1\}^\N$ by $\iota(x) = \iota_n(x)$ for any $n$ with $x\in \B(x,2^n d)$.
        This is an isometry because, given $x,y\in X$, there exists an $n\in\N$ with $x,y\in \B(a,2^n d)$ and so
        \begin{equation*}
                d(\iota(x),\iota(y))= d(\iota_n(x),\iota_n(y)) =d(x,y).
        \end{equation*}
        It is surjective because
        \begin{equation*}
                \B(a,d)\times \{0,1\}^{n-1}\times \{0\}^\N
        \end{equation*}
        contains a ball of radius $d2^n$ around $(a,0,0,\ldots)$.

        To construct each $\iota_n$, define
        \begin{equation*}
                \iota_1(x) = (x,0,0,\ldots),
        \end{equation*}
        an isometry.
        Now suppose that we have defined $\iota_n$.
        By inductively applying \Cref{prop:double-pair}, there exists a $b'\in X$ with $d(a,b')=d2^n$ such that $a,b'$ is a Besicovitch pair.
        By \Cref{lem:bpair-properties} \cref{item:bc-double-ball},
        \begin{equation*}
                \B(a,2^{n+1}d) = \B(a,2^n d) \cup \B(b',2^n d).
        \end{equation*}
        Define $\iota_{n+1}$ to equal $\iota_n$ on $\B(a,2^n d)$ and, for any $y\in \B(b',2^n d)$,
        \begin{equation*}
                \iota_{n+1}(y) = I(\iota_{n}(\iota_{ab'}(x))),
        \end{equation*}
        where $I$ exchanges the $n^\text{th}$ coordinate of the $\{0,1\}^\N$ factor.
        Since $\iota_{ab'}$ is an isometry, $\iota_{n+1}$ is an isometry on each component of $\B(a,2^{n+1}d)$, we must show it is an isometry on the whole domain.

        To this end, let $x\in \B(a,2^nd)$ and $y\in \B(b',2^nd)$.
        By construction, both $\iota_n(x)$ and $\iota_n(\iota_{ab'}(y))$ have a $0$ in the $n^\text{th}$ coordinate of the $\{0,1\}^\N$ factor.
        Therefore
        \begin{equation*}
                d(\iota_{n+1}(x),\iota_{n+1}(y)) = d(\iota_n(x),I(\iota_n(\iota_{ab'}(y)))) = 2^n d + d(\iota_n(x),\iota_n(\iota_{ab'}(y)))
        \end{equation*}
        By induction $\iota_n$ is an isometry and so
        \begin{equation*}
                d(\iota_{n+1}(x),\iota_{n+1}(y)) = 2^n d + d(x,\iota_{ab'}(y)) = d(x,y)
        \end{equation*}
        by \Cref{lem:bpair-properties} \cref{item:bc-l1-dist}.
\end{proof}
\section{Classifying 1-uniform metric measure spaces}
\label{sec:classification}

The relationship between connected compact sets of finite Hausdorff measure and Lipschitz curves  is well known.
We will use the following, see, for example, \cite[Theorems 4.4.20, 4.2.1, 4.4.2]{ambtilli}.
\begin{theorem}\label{thm:curve-connected}
        Let $X$ be a compact connected metric space with $\H(X)<\infty$.
        Then for any $x,y\in X$ there exists an injective Lipschitz curve $\gamma_{xy} \colon [0,L] \to S$ with $\gamma_{xy}(0)=x$, $\gamma_{xy}(L)=y$ and $\H(\gamma_{xy})$ minimal amongst all curves joining $x$ to $y$.

        Moreover, we may suppose that $\gamma_{xy}$ is \emph{parametrised by arc-length}.
        That is, $\gamma_{xy}$ is 1-Lipschitz and, for any $s\leq t \in [0,L]$,
        \begin{equation*}
                \H(\gamma_{xy}([s,t])) = t-s.
        \end{equation*}
\end{theorem}

We will be interested in \emph{parametrisations} of Lipschitz curves.
For convenience, we will refer to both the function $\gamma \colon I \to X$ or its image $\gamma(I)$ by simply $\gamma$.

The previous theorem is often used to show that a compact connected metric space of finite $\H$ measure is the image of a Lipschitz curve.
We prove a variant for uniform metric measure spaces.
\begin{lemma}\label{lem:compact-connected-contained-in-curve}
        Let $(K,d)$ be a (non-empty) compact connected metric space with $\H(K)<\infty$.
        Suppose that, for any injective $\gamma\colon I \to K$ that is parametrised by arc-length, and any $[c-r,c+r]\subset I$,
        \begin{equation*}
                \B(\gamma(c),r) = \gamma([c-r,c+r]).
        \end{equation*}
        Then $K$ is the image of a Lipschitz curve $\gamma \colon I \to K$.
        On any proper subinterval, $\gamma$ is injective and parametrised by arc-length.
\end{lemma}

\begin{proof}
        If $K$ consists of a single point, we are done.
        Otherwise, let $x,y$ be distinct points in $K$.
        By \Cref{thm:curve-connected} there exists an injective Lipschitz $\gamma_1\colon I_1\ \to K$ of positive length joining $x$ to $y$ that is parametrised by arc-length.

        We now proceed to inductively construct an increasing sequence of compact intervals $I_n$ and sequence of Lipschitz curves 
        \begin{equation*}
                \gamma_n \colon I_n \to K
        \end{equation*}
        each of which are parametrised by arc-length, are equal on their common domain, and satisfy
        \begin{equation}\label{eq:curve-eats-measure}
                \H(\gamma_n) \geq \sum_{i=1}^{n-1} \max \{d(z,\gamma_i) : z\in K\}.
        \end{equation}

        For $n \in \N$ suppose that we have defined such $\gamma_1,\ldots,\gamma_n$.
        If $\gamma_n$ is the whole of $K$ then we stop and define all subsequent $\gamma_i = \gamma_n$.
        Note that \cref{eq:curve-eats-measure} is satisfied.
        Otherwise, let $z_n \in K \setminus \gamma_n$ satisfy
        \begin{equation}\label{eq:max-distance-from-curve}
                d(z_n,\gamma_n) = \max \{d(z,\gamma_n) : z\in K\} >0,
        \end{equation}
        which exists by the compactness of $K$ and $\gamma_n$.
        Also, for $e,e'$ the end points of $\gamma_n$, choose $y \in \{e,e'\}$ such that
        \begin{equation}\label{eq:shortest-curve}
                \H(\gamma_{y z}) \leq \H(\gamma_{e z}), \H(\gamma_{e' z}).
        \end{equation}

        Observe that $\gamma_n \cap \gamma_{y z} = \{y\}$.
        Indeed, by applying \Cref{lem:uniform-curve-properties} with $c=(\gamma_n^{-1}(e)+ \gamma_n^{-1}(e'))/2$ and $r=|\gamma_n^{-1}(e)-\gamma_n^{-1}(e')|/2$, $B(\gamma_n(c),r) = \gamma_n$.
        Therefore, since $\gamma_n$ is injective, if $\gamma_n\cap \gamma_{y z}$ contains a point other than $y$, it must also contain the other end point, $y'$, of $\gamma_n$, and then continue to join $y'$ to $z$.
        This would mean that $\H(\gamma_{y z}) > \H(\gamma_{y' z})$, which contradicts \cref{eq:shortest-curve}.

        Construct an injective 1-Lipschitz curve $\gamma_{n+1}$ by concatenating $\gamma_n$ and $\gamma_{yz}$ in the compatible way.
        Note that, since $\gamma_n \cap \gamma_{yz}$ is a single point, and hence a set of measure zero, $\gamma_{n+1}$ is parametrised by arc-length.
        Since both $\gamma_n$ and $\gamma_{yz}$ are subsets of $K$, so is $\gamma_{n+1}$.
        Also, since $\gamma_{yz}$ is parametrised by arc-length, $\H(\gamma_{yz}) \geq d(y,z)=d(z_n,\gamma_n)$.
        By combining \cref{eq:max-distance-from-curve} and the induction hypothesis \cref{eq:curve-eats-measure} on $\gamma_n$, we see that \cref{eq:curve-eats-measure} is true for $\gamma_{n+1}$.
        This completes the inductive step for $\gamma_{n+1}$.

        Observe that, since each $\gamma_n$ is parametrised by arc-length,
        \begin{equation}\label{eq:curve-eats-2}
                \operatorname{length}(I_n) =\H(\gamma_n) \leq \H(\B(x,r)).
        \end{equation}
        Since the $I_n$ are increasing, $I=\cup_n I_n$ is an interval of finite length and by \Cref{lem:increasing-define-function} we obtain an injective $\gamma \colon I \to K$ defined by $\gamma(t) = \gamma_n(t)$ whenever $t\in I_n$.
        It is parametrised by arc-length.
        By \cref{eq:curve-eats-2}, the closure of $\gamma$ is the whole of $K$.
        However, the closure of $\gamma$ is simply the unique extension of $\gamma$ to the closure $\bar I$ of $I$.
        By translating $I$, we may suppose $\bar I = [-l,l]$ for some $l>0$.
        Since $[-l,l]  \setminus I$ is at most $\pm l$, this curve is parametrised by arc-length and injective on $I$.

        Now suppose $c\in (-l,l)$, so that there exists an $r>0$ such that $[c-r,c+r] \subset (-l,l)$.
        For any sufficiently small $\epsilon>0$, $\gamma$ restricted to $[-l+\epsilon,c+r]$ is injective and parametrised by arc-length.
        Therefore, by hypothesis, $\gamma(-l+ \epsilon) \not\in B(\gamma(c),r)$.
        Since $\epsilon>0$ is arbitrary, $\gamma(l)\not\in B(\gamma(c),r)$ and hence $\gamma(l)\neq \gamma(c)$.
        Similarly, $l$ cannot be mapped to the same point as $c$, and so $\gamma$ is injective on any subinterval.
\end{proof}
The uniform condition imposes strong geometric constraints on a Lipschitz curve.
\begin{lemma}\label{lem:uniform-curve-properties}
        Let $(X,d,\mu)$ be a uniform metric measure space.
        \begin{enumerate}
                \item \label{item:uniform-curve-injective} For any injective $\gamma\colon I \to X$ that is parametrised by arc-length and any $[c-r,c+r]\subset I$,
        \begin{equation*}
                \B(\gamma(c),r) = \gamma([c-r,c+r])
        \end{equation*}
        and $d(\gamma(c),\gamma(c\pm r)) = r$.
\item \label{item:uniform-curve-surjective} For any surjective $\gamma \colon I \to \B(x,r)$ that is parametrised by arc-length, after possibly translating $I$, $I=[-r,r]$ and either $\gamma(r)=\gamma(-r)$ or $\gamma(0)=x$.
        \end{enumerate}
\end{lemma}

\begin{proof}
        By multiplying $\mu$ by a constant, we may suppose that $\mu(B(x,r))=2r$ for each $x\in X$ and $r>0$.
        Observe that $\gamma(I)$ is contained within an open ball $B$, that $\mu|_{B}(B(x,r)) = 2r$ for all $x\in B$ and all sufficiently small $r>0$ and that $\H(B) <\infty$.
        Therefore, by \Cref{cor:rectifiable-regular}, $\mu|_{\gamma(I)} = \H|_{\gamma(I)}$.

        Note that for any $[c-r,c+r]\subset I$, $\gamma|_{[c-r,c+r]}$ is also parametrised by arc-length.
        Therefore,
        \begin{equation*}
                \mu(\gamma((c-r,c+r))) = \H(\gamma((c-r,c+r)) = 2r.
        \end{equation*}
        Also, since $\gamma$ is 1-Lipschitz, $\gamma((c-r,c+r)) \subset B(\gamma(c),r)$.
        Thus, by \Cref{lem:uniform-properties} \cref{item:closure-of-open}, $\B(\gamma(c),r) = \gamma([c-r,c+r])$.

        Note that, for any $[c-r,c+r] \subset [-l,l]$, since $\gamma$ is 1-Lipschitz, $d(\gamma(c),\gamma(c\pm r)) \leq r$.
        Also, for any $\epsilon>0$, since $\gamma$ is injective on $[c-(r-\epsilon), c+(r-\epsilon)]$,
        \begin{equation*}
                \gamma(c \pm r) \not\in \gamma([c-(r-\epsilon), c+ (r-\epsilon)]) = \B(\gamma(s),r- \epsilon),
        \end{equation*}
        so that $d(\gamma(c),\gamma(c \pm r)) \geq r-\epsilon$.
        Since $\epsilon>0$ is arbitrary, $d(\gamma(c),\gamma(c \pm r))=r$, as required.

        Further, if $\gamma \colon I \to \B(x,r)$ is surjective and parametrised by arc-length, then
        \begin{equation*}
                2r = \mu(\gamma(I)) = \H(\gamma(I)) = \operatorname{length}(I)
        \end{equation*}
        and so by translating, we may suppose that $I= [-r,r]$.
        If $\gamma(r) \neq \gamma(-r)$ then $\gamma$ is injective and so we must have $d(x,\gamma(\pm r))=r$.

        Indeed, if $d(x,\gamma(r))<r$ then there exists an $\epsilon>0$ such that $B(\gamma(r),\epsilon) \subset B(x,r)$ and such that $\B(\gamma(\pm r),\epsilon)$ are disjoint.
        By the first part of the lemma, $d(\gamma(r),\gamma(0))=r$ and $\gamma([r-\epsilon,r+\epsilon]) = \B(\gamma(0),r-\epsilon)$, and so by the triangle inequality, $\gamma([r-\epsilon,r+\epsilon])$ and $B(\gamma(r),\epsilon)$ are disjoint.
        Therefore, $B(\gamma(r),\epsilon) = \gamma((r-\epsilon,r])$.
        By hypothesis, $\mu(B(\gamma(r),\epsilon)) = 2r$ but, since $\gamma$ is parametrised by arc-length, $\mu(\gamma((r-\epsilon,r])) = r$, a contradiction.
        Thus $d(\gamma(\pm),x)=r$.
        Since $\gamma$ is 1-Lipschitz, the only possibility is $x=\gamma(0)$.
\end{proof}

We begin our classification of uniform metric measure spaces.
\begin{theorem}\label{thm:uniform-connected-is-R}
A connected uniform metric measure space is isometric to $\R$.
\end{theorem}

\begin{proof}
        Fix a point $x_0\in X$.
        First note that, since $X$ is connected, for any $r>0$, $\B(x_0,r)$ is also connected.
        Indeed, if $\B(x_0,r)$ is disconnected, then by \Cref{prop:existence-of-pairs}, there exists a Besicovitch pair $a,b\in \B(x_0,r)$.
        By \Cref{lem:bpair-properties} \cref{item:bc-components-are-open}, $\B(a,d(a,b))$ is a non-trivial open and closed subset of $X$, a contradiction.
        
        Now, for any $r>0$, \Cref{lem:compact-connected-contained-in-curve} and \Cref{lem:uniform-curve-properties} \cref{item:uniform-curve-surjective} give a surjective Lipschitz $\gamma_r \colon [-r,r] \to \B(x_0,r)$ that is parametrised by arc-length and injective except possibly at the end points.
        Note, this means that $\gamma_{r/2}$ cannot map ${-r/2,r/2}$ to the same point, since $\gamma_r$ does not contain a closed loop of measure $r$.
        Therefore, $\gamma_{r/2}(0)=x$ by \Cref{lem:uniform-curve-properties} \cref{item:uniform-curve-surjective}.
        Also, by \Cref{lem:uniform-curve-properties} \cref{item:uniform-curve-injective},
        \begin{equation*}
                \gamma^{r/4} := \gamma_{r/2}|_{[-r/4,r/4]} \colon [-r/4,r/4] \to \B(x_0,r/4)
        \end{equation*}
        is an isometric surjection.
        Note that, for any $r<r'$, 
        \begin{equation*}
                (\gamma^{r'})^{-1} \circ \gamma^r \colon [-r,r] \to [-r,r]
        \end{equation*}
        is an isometry and so equals $\pm \operatorname{Id}$.
        Therefore, by pre-composing by multiplication by $-1$ if necessary, the $\gamma^r$ agree on their common domain.
        Thus, by \Cref{lem:increasing-define-function}, there exists an isometry
        \begin{equation*}
                \gamma \colon \R = \bigcup_{n\in\N} [-n,n] \to \bigcup_{n\in\N} \B(x_0,n) = X,
        \end{equation*}
        as required.
\end{proof}

A similar type of argument handles the case that a component of a Besicovitch pair is connected.
\begin{lemma}\label{lem:connected-bpair-is-T}
        Suppose that $a,b$ is a Besicovitch pair in a uniform metric measure space and that $\B(a,d(a,b))$ is connected.
        Then $\B(a,d(a,b))$ is isometric to $d(a,b)\T$.
\end{lemma}

\begin{proof}
        Let $d=d(a,b)$.
        By \Cref{lem:compact-connected-contained-in-curve}, $\B(a,d)$ is the image of a Lipschitz curve $\gamma\colon [-d,d] \to \B(a,d)$ that is parametrised by arc-length and injective except possibly at the end points.
        First observe that necessarily, $\gamma(-d)=\gamma(d)$.
        Indeed, if $\gamma(-d)\neq \gamma(d)$ then by \Cref{lem:uniform-curve-properties} \cref{item:uniform-curve-injective}, $\gamma(0)=a$.
        However, by \Cref{lem:bpair-properties} \cref{item:bc-components-are-open}, for any $x\in \B(a,d)$, $\B(x,d)=\B(a,d)$.
        By applying \Cref{lem:uniform-properties} \cref{item:uniform-curve-injective} again, $\gamma(0)=x$, so that $\B(a,d)=\{a\}$ and hence $\gamma(-d)=\gamma(d)$, a contradiction.

        Thus, we may periodically extend $\gamma$ to a 1-Lipschitz function defined on $\R$ with period $2d$.
        Moreover, for any $t \in [-d,d]$ and $0\leq r \leq d$, $\gamma$ restricted to $[t-r,t+r]$ is parametrised by arc-length and so, by \Cref{lem:uniform-curve-properties} \cref{item:uniform-curve-injective}, $d(\gamma(t),\gamma(t\pm r))=r$.
        That is, when considering $\gamma\colon d\T\to \B(a,d)$, $\gamma$ is an isometry.
\end{proof}

The final lemma we require to complete the classification is to show that the existence of a single Besicovitch pair implies that all points form a Besicovitch pair.
\begin{lemma}\label{lem:change-basepoint-of-bpairs}
        Let $(X,d,\mu)$ be a uniform metric measure space and $a,b$ a Besicovitch pair and $x\in X$.
        Then there exists a $x'\in X$ with $d(a,b)=d(x,x')$ such that $x,x'$ is a Besicovitch pair.

        In particular, if $a_n,b_n\in X$ a sequence of Besicovitch pairs, there exist $a\in X$ and a sequence $b'_n\in X$ such that $d(a,b'_n) = d(a_n,b_n)$ and $a,b'_n$ is a Besicovitch pair for each $n\in \N$.
\end{lemma}

\begin{proof}
Identify $X$ with its isometric image in $\B(a,d(a,b)) \times \{0,1\}^\N$ given by \Cref{thm:bp-isometry-to-l1}.
With this identification, $x$ is contained in an isometric copy $B_1$ of $\B(a,d(a,b))$, with a corresponding isometric copy $B_2$ of $\B(b,d(a,b))$, which form a Besicovitch pair with a corresponding isometry $\iota$ between $B_1$ and $B_2$ given by \Cref{def:neighbour}.
By \Cref{lem:bpair-properties}, $x$ and $\iota(x)$ are a Besicovitch pair with $d(x,\iota(x))=d(a,b)$, as required.

The statement about the sequences $a_n,b_n$ follows by applying the first part of the lemma to $x=a_1$.
\end{proof}

We now proceed to prove the classification of uniform metric measure spaces.
\begin{theorem}\label{thm:bounded-connected-is-T}
        Let $(X,d,\mu)$ be a uniform metric measure space.
        Suppose that
        \begin{equation*}
                0< \delta = \inf \{d(a,b): a,b \text{ a Besicovitch pair in } X\} <\infty.
        \end{equation*}
        Then $X$ is isometric to $\delta \mathcal T$.
\end{theorem}

\begin{proof}
Fix $x_0 \in X$.
We first show that there exists a Besicovitch pair $a,b\in X$ with $d(a,b)=\delta$.

For any integer $n > 1/\delta$, by hypothesis and \Cref{lem:change-basepoint-of-bpairs}, there exists an $a\in X$ and a sequence $b_n\in X$ such that let $a,b_n$ is a Besicovitch pair with $\delta \leq d(a,b_n) < \delta +1/n$.
In particular, by compactness, we may suppose that $b_n\to b$.

Now suppose that $x\in B(a,\delta)$ and $y\in B(b,\delta)$.
Then there exists an $n\in \N$ such that $x\in B(a_n,\delta)$ and $y\in B(a_n,\delta)$.
Since $a_n,b_n$ are a Besicovitch pair, we must have $d(x,y)\geq d(a_n,b_n) \geq d(a_n,b_n) \geq \delta$, so that $a,b$ is a Besicovitch pair.

Finally, note that $\B(a,r)$ is connected for all $0<r<\delta$, since otherwise \Cref{prop:existence-of-pairs} would give a Besicovitch pair separated by a distance strictly less that $\delta$.
Therefore, $\B(a,\delta)$, being the closure of a union of connected sets all containing $a$, is connected.
Thus, by \Cref{lem:connected-bpair-is-T}, $\B(a,\delta)$ is isometric to $\delta \T$ and hence by \Cref{thm:bp-isometry-to-l1}, $X$ is isometric to $\delta \mathcal T$.
\end{proof}

\begin{theorem}\label{thm:totally-disconnected-is-S}
        Let $(X,d,\mu)$ be a uniform metric measure space.
        Suppose that
        \begin{equation*}
                \inf\{d(a,b): a,b \text{ a Besicovitch pair in } X\}=0.
        \end{equation*}
        Then $X$ is isometric to a scaled copy of $\S$.
\end{theorem}

\begin{proof}
        By \Cref{lem:change-basepoint-of-bpairs}, there exist $a\in X$ and, for each $n\in \N$, a $b_n\in X$ such that $d(a,b_n)<1/n$ and $a,b_n$ is a Besicovitch pair.
        Let $d(a,b_n)=\lambda_n 2^{-m_n}$ with $1\leq \lambda_n \leq 2$.
        By taking a subsequence if necessary, we may suppose that $\lambda_n\to \lambda \in [1,2]$.

        By \Cref{thm:bp-isometry-to-l1}, for each $n\in \N$, $X$ is isometric to $\B(a,d(a,b_n)) \times \{0,1\}^\N$ with the metric
        \begin{equation*}
                d((x,s),(y,t)) = d(x,y) + \lambda_n 2^{-m_n}\sum_{i\in \N} 2^i |s_i-t_i|.
        \end{equation*}
        We now identify $X$ with this isometric image and define the following map on $X$:
        \begin{align*}
                \iota_n \colon X &\to \lambda \mathcal S\\
                \iota_n((x,s)) &= (\ldots, 0, 0, s_1, s_2, s_3, \ldots),
        \end{align*}
        where the coordinates of $s$ begin in the $-m_n +1$ coordinate of $\lambda\mathcal S$.
        Note that for any $(x,s),(y,t) \in X$,
        \begin{equation*}
                d(\iota_n(x,s),\iota_n(y,t)) = \lambda 2^{-m_n} \sum_{i\in \N} 2^i |s_i -t_i| = \frac{\lambda}{\lambda_n} (d((x,s),(y,t)) - d(x,y)),
        \end{equation*}
        so that $\iota_n$ is $\lambda/\lambda_n$-Lipschitz and
        \begin{equation}\label{eq:iota-close-to-isometry}
                d(\iota_n((x,s)),\iota_n(y,t)) \geq \frac{\lambda}{\lambda_n} d((x,s),(y,t)) - 2d(a,b_n)
        \end{equation}
        for each $(x,s),(y,t)\in X$.
        By subtracting $\iota_n(a)$ from $\iota_n$, we may suppose that $\iota_n(a)=0$ for each $n\in\N$.

        For each $j\in \N$, $\B(a,2^j)$ is compact and
        \begin{equation*}
                \iota_n(\B(a,2^j)) \subset \B(0,L2^j/L_n) \subset \B(0,2^{m+1}),
        \end{equation*}
        which is also compact.
        Therefore, by Arzelà-Ascoli, after possibly taking a subsequence, there exists $\iota\colon \B(a, 2^j) \to \B(0,2^j)$ such that $\iota_n\to \iota$ uniformly.
        Necessarily, since $\lambda/\lambda_n \to 1$, $\iota$ is 1-Lipschitz, and since $d(a,b_n)\to 0$, \cref{eq:iota-close-to-isometry} shows that $\iota$ is an isometry onto its image.
        Further, since $\B(a,2^j)$ contains
        \begin{equation*}
                \{(a,s): s = (\ldots, \pm 1,\pm 1, 0, 0, \ldots)\},
        \end{equation*}
        where the $\pm 1$ end at the $j-1$ coordinate,
        $\iota_n(\B(a,2^j))$ contains
        \begin{equation*}
                N_n=\{s: s=(\ldots, 0, 0, \pm 1, \pm 1, \ldots, \pm 1, 0, 0, \dots)\},
        \end{equation*}
        where the $\pm 1$ begin at the $-m_n +1$ coordinate and end at the $j-1$ coordinate.
        Since the $N_n$ form an increasing sequence of $2^{-m_n+1}$ nets of $\B(0,2^{j-1})$ and since $\iota(\B(0,2^j)$ is closed, $\iota(\B(a,2^j))$ contains $\B(0,2^{j-1})$.

        Finally, taking a convergent diagonal subsequence of such $\iota$ over $j\to \infty$ gives the required isometry.
\end{proof}

\begin{proof}[Proof of \Cref{thm:main-uniform}]
        Let $(X,d,\mu)$ be uniform.
        If $X$ is connected, then by \Cref{thm:uniform-connected-is-R}, $X$ is isomorphic to $\R$.
        Second, if $X$ is disconnected, then there exists an $x\in X$ and $r>0$ such that $\B(x,r)$ is disconnected and so by \Cref{prop:existence-of-pairs} there exists a Besicovitch pair in $X$.
        If
        \begin{equation*}
                \inf\{d(a,b): a,b \text{ a Besicovitch pair in } X\}>0
        \end{equation*}
        then by \Cref{thm:bounded-connected-is-T} $X$ is isometric to a scaled copy of $\mathcal T$.
        On the other hand, if
        \begin{equation*}
                \inf\{d(a,b): a,b \text{ a Besicovitch pair in } X\}=0,
        \end{equation*}
        then $X$ is isometric to a scaled copy of $\mathcal S$ by \Cref{thm:totally-disconnected-is-S}.
\end{proof}

\section{Tangents, rectifiability and 1-regular metric measure spaces}
\label{sec:tangents}

In this section we prove \Cref{regular_tan}.
By the results in the previous sections, we will see that it suffices to show that $\T$ does not belong to $\Tan(\mu,x)$ for almost every $x$.
This approach is similar in spirit to the one of Preiss \cite{preiss}, but is substantially simpler due to our complete characterisation of 1-uniform metric measure spaces.

\begin{lemma}
	\label{scaling-dGHs}
	Let $(X,d,\mu,x)$ be regular, $\delta>0$ and $\lambda\geq 1$.
	For $\mu$-a.e. $x\in X$ there exists $r_1>0$ such that, for any $0<r<r_1$, if $(Y,\rho,\nu,y) \in \U$ and
	\begin{equation*}
		\dGHs(T_r(X,d,\mu,x),(Y,\rho,\nu,y)) < \delta,
	\end{equation*}
	then
	\begin{align*}
		\dGHs(T_{\lambda r}(X,d,\mu,x),T_{\lambda}(Y,\rho,\nu,y))
		< \lambda \delta.
	\end{align*}
\end{lemma}

\begin{proof}
	Fix $\epsilon>0$ to be chosen later.
	For $\mu$-a.e. $x\in X$, there exist $\theta,r_0>0$ such that
	\[(1-\epsilon)\theta r < \mu(B(x,r)) < (1+\epsilon)\theta r\]
	for all $0<r<r_0$.
	Fix such an $x$.

	Let $(Z,\zeta)$ be a metric space and $z\in Z$ such that there exist isometric embeddings of $T_r X,Y$ into $Z$, mapping $x$ and $y$ to $z$ such that
	\begin{equation*}
		\dKR_z^{\frac{1}{\delta},\frac{1}{\delta}}
		\left(
		\frac{\mu}{\mu(B(x,r))},
		\nu
		\right)
		< \delta.
	\end{equation*}
	Consider $\widetilde Z:=(Z,\zeta/\lambda)$.
	Then, for any $L,\rho>0$,
	\begin{equation*}
		\widetilde\dKR_z^{\frac{L}{\lambda},\frac{\rho}{\lambda}} 
		\leq
		\widetilde\dKR_z^{\lambda L,\frac{\rho}{\lambda}}
		=
		\dKR_z^{L,\rho},
	\end{equation*}
	where a tilde denotes a quantity measured in $\widetilde Z$.
	In particular,
	\begin{equation*}
		\widetilde\dKR_z^{\frac{1}{\lambda\delta},\frac{1}{\lambda\delta}}
		\left(
		\frac{\mu}{\mu(B(x,r))},
		\nu
		\right)
		\leq
		\dKR_z^{\frac{1}{\delta},\frac{1}{\delta}}
		\left(
		\frac{\mu}{\mu(B(x,r))},
		\nu
		\right)
		< \delta
	\end{equation*}
	and so
	\begin{align*}
		\widetilde \dKR_z^{\frac{1}{\lambda \delta}\frac{1}{\lambda \delta}}\left(
			\frac{\mu}{\mu(B(x,\lambda r))}
			,
		\frac{\nu}{\nu(B(y,\lambda))}
		\right)
		&\leq
		\frac{1}{\lambda}\widetilde \dKR_z^{\frac{1}{\lambda \delta}\frac{1}{\lambda \delta}}\left(
			\frac{\mu}{\mu(B(x,r))}
			,
		\nu
		\right)
		\\&\quad +
		\widetilde \dKR_z^{\frac{1}{\lambda \delta}\frac{1}{\lambda \delta}}\left(
			\frac{\mu}{\mu(B(x,\lambda r))}
			,
		\frac{\mu}{\lambda\mu(B(x,r))}
		\right)
		\\&<
		\frac{\delta}{\lambda}
		\\&\quad + 
		\frac{
			|\lambda \mu(B(x,r))-\mu(B(x,\lambda r))|
		}{
			\lambda \mu(B(x,\lambda r))\mu(B(x,r))
		}
		\mu(B(x,(\lambda\delta)^{-1}r)).
	\end{align*}
	Provided we pick $r_1< (\lambda \delta) r_0$ and require $0<r<r_1$, this quantity is bounded above by
	\begin{equation*}
		\frac{\delta}{\lambda}
		+
		\frac{
			2\lambda \theta r \epsilon
		}{
			(1-\epsilon)^2 \lambda^2 \theta^2 r^2
		}
		\frac{(1+\epsilon)\theta r}{\lambda \delta}.
	\end{equation*}
	Provided $\epsilon>0$ is sufficiently small, depending only upon $\delta$ and $\lambda$, this quantity is bounded above by
	\begin{equation*}
		\frac{\delta}{\lambda} + \frac{3\epsilon}{\lambda^2 \delta} < \lambda \delta.
	\end{equation*}
	The conclusion follows.
\end{proof}

We now prove estimates on the value of $\dGHs$ between elements of $\U$.
This is the only part in our argument that depends on the exact form of $\dKR$ in the definition of $\dGHs$.
It is possible to formulate a proof of \cref{regular_tan} that is independent of this choice, but it is far more involved.
\begin{lemma}
	\label{d_T_to_R}
	For any $r>0$,
	\begin{equation}
		\label{d_T_to_R-eq}
		\min \left\{\frac{r}{4}, \frac{1}{2} \right\} \leq \dGHs(T_r\T, \R) \leq 2r.
	\end{equation}
	and $\dGHs(T_r\S,\R) \geq 1/8$.
\end{lemma}

\begin{proof}
	The ball $B(0,1/2r)\subset T_r\T$ is isometric to the interval $(-1/2r,1/2r)\subset \R$, each possessing the same measure, and so the right hand inequality of the first statement holds.

	In order to prove the other inequalities, we first show the following.
	Let $r>0$ and suppose that $(Y,\rho,\nu,y)$ is 1-uniform and $Q\subset Y$ has diameter $r$, contains $y$ and satisfies
	\[\{x\in Y: \rho(x,Q)<r\} \subset Q.\]
	We claim that $\dGHs(\R,Y) \geq \min\{r/4,1/2\}$.

	To this end, suppose that $\R$ and $Y$ are isometrically embedded into some metric space $Z$, with each distinguished point mapped to the same point $0$.
	Note that, since each $\dKR_0^{L,\rho}$ scale linearly for (1-)uniform measures, it suffices to show that, in the case $r=1$, $\dKR_0^{4,4}(\H|_Y,\H|_R) \geq 1/4$.

	Let $g\colon Z\to [-1,1]$ be the $4$-Lipschitz function satisfying
	\[ g(z) = \begin{cases} 0 & z\in Q\\
		0 & z\not\in B(Q,2) \\
		1 & z\in B(Q,1.75)\setminus B(Q,1.25)
	\end{cases}.\]
	Then $\spt g\subset B(0,3)$ and so
	\begin{align*}
		\dKR_0^{4,3}(\H|_{Y},\H|_{\R}) &\geq \int_{B(0,3)} g\d(\H|_{\R}-\H|_{Y})
		\\ &\geq \int_{B(Q,1.75)\setminus B(Q,1.25)} 1 \d\H|_{\R} \geq 1.
	\end{align*}
	Consequently,
	\begin{equation*}
		\dKR_0^{4,4}(\H|_{Y},\H|_{\R}) \geq 1/4,
	\end{equation*}
	proving the claim.

	The first inequality of the conclusion now follows by applying the claim to $Q= C(0,r) \subset T_r \T$.
	For the second part, let $r>0$ and let $i\in\Z$ be such that $2^{-(i+1)} \leq r < 2^{-i}$.
	Applying the claim to $Q= C(0,2^{i}r) \subset T_r \S$ gives $\dGHs(T_r\S,\R) \geq 2^{i}r/4 \geq 1/8$.
\end{proof}

We nom prove a key geometric property.
It will be used to show that, if $\R$ is a tangent space, then it is the only tangent space.
\begin{proposition}
	\label{flat_at_inf}
	Let $(X,d,\mu)$ be regular, $\lambda \geq 1$ and $0<\epsilon<1/24$ and suppose that $x\in X$ satisfies the conclusion of \Cref{scaling-dGHs}.
	Suppose also that, for $0<r<r_1$, 
	\begin{equation}
		\label{p-close-R}
		\dGHs(T_r(X,d,\mu,x),\R) < \epsilon
	\end{equation}
	and, for some $U\in\U$,
	\begin{equation}
		\label{p-close-T}
		\dGHs(T_{r/\lambda}(X,d,\mu,x),U) < \frac{\epsilon}{\lambda}.
	\end{equation}
	Then
	\begin{equation*}
		\dGHs(T_{r/\lambda} (X,d,\mu,x),\R) < \frac{25}{\lambda}\epsilon.
	\end{equation*}
\end{proposition}

\begin{proof}
	Note that, for any $r,r'>0$, $T_rT_{r'}=T_{rr'}=T_{r'r}$.
	Therefore, \eqref{p-close-T} and \Cref{scaling-dGHs} imply that
	\begin{equation*}
		\dGHs(T_r(X,d,\mu,x),T_{\lambda} U) < \epsilon.
	\end{equation*}
	Hence \eqref{p-close-R} and the triangle inequality imply
	\begin{equation}
		\label{U_close_R}
		\dGHs(\R,T_\lambda U) < 3\epsilon < \frac{1}{8}.
	\end{equation}
	Hence the second conclusion of \Cref{d_T_to_R} implies that $U=T_{\rho} \T$ for some $\rho>0$ or $U=\R$.

	If the latter holds then \eqref{p-close-T} directly implies the conclusion.
	If the former holds then \eqref{U_close_R} becomes
	\begin{equation*}
		\dGHs(\R,T_{\lambda \rho} \T) < 3\epsilon
	\end{equation*}
	and hence \eqref{d_T_to_R-eq} implies that
	\begin{equation*}
		\frac{\lambda \rho}{4} \leq \dGHs(\R,T_{\lambda \rho} \T) < 3\epsilon.
	\end{equation*}
	Thus \eqref{d_T_to_R-eq} also implies
	\begin{equation*}
		\dGHs(U,\R) = \dGHs(T_\rho\T,\R) < \frac{24}{\lambda} \epsilon.
	\end{equation*}
	Therefore, \eqref{p-close-T} and the triangle inequality gives
	\begin{align*}
		\dGHs(T_{r/\lambda}(X,d,\mu,x),\R) &\leq 	\dGHs(T_{r/\lambda} (X,d,\mu,x),U) + \dGHs(U,\R)\\
		&< \frac{1}{\lambda}\epsilon + \frac{24}{\lambda}\epsilon.
	\end{align*}
\end{proof}

\begin{corollary}
	\label{tan_is_R}
	If $(X,d,\mu)$ is regular then for $\mu$-a.e. $x\in X$, either $\R\not\in\Tan(\mu,x)$ or $\Tan(\mu,x)=\{\R\}$.
\end{corollary}

\begin{proof}
	Suppose that $R\in \Tan(\mu,x)$ and let $0<\epsilon<1/24$.
	By \Cref{tan-of-unif} and \Cref{tangents_exist}, $\mu$-a.e. $x\in X$ satisfies
	\[\Tan(\mu,x)\subset \U\]
	and
	\[\dGHs(T_{r}(\mu,x),\Tan(\mu,x)) < \frac{\epsilon}{25}\]
	for all $0<r<r_x$, for some $r_x>0$.
	That is,
	\begin{equation}
		\label{close_to_unif}
		\dGHs(T_r(\mu,x),\U) < \frac{\epsilon}{25}
	\end{equation}
	for all $0<r<r_x$.
	Fix such an $x$ that also satisfies the conclusion of \Cref{scaling-dGHs}.

	By the definition of $\Tan(\mu,x)\ni \R$, there exists some $0<r_0<r_x$ such that
	\[\dGHs(T_{r_0}(\mu,x), \R) < \epsilon.\]
	\Cref{close_to_unif} implies that the hypotheses of \Cref{flat_at_inf} are satisfied for $\lambda=25$ and so we may deduce that
	\[\dGHs(T_{r_0/25}(\mu,x),\R) < \epsilon.\]
	Repeating inductively, we see that
	\[\dGHs(T_{25^{-i}r_0}(\mu,x),\R) < \epsilon\]
	for all $i\in\N$.
	Now let $0<r<r_0$ and let $i\in\N$ be such that $2^{-(i+1)}r_0 < r\leq 2^{-i} r_0$.
	Applying \Cref{scaling-dGHs} gives
	\begin{equation*}
		\dGHs(T_r(\mu,x),\R) \leq 25 \dGHs(T_{2^{-i} r_0}(\mu,x), \R) < 25\epsilon.
	\end{equation*}
	Since $\epsilon>0$ is arbitrary, this implies that $T_r(\mu,x) \to \R$ as $r\to 0$, as required.
\end{proof}

Finally we prove our main result on regular measures.
\begin{proof}[Proof of \Cref{regular_tan}].
	By \Cref{tan-of-unif}, for $\mu$-a.e. $x\in X$, $\Tan(\mu,x)\subset \U$.
	\Cref{tan_is_R} then implies, for $\mu$-a.e. $x\in X$, that
	\begin{equation}
		\label{final_eq}
		\Tan(\mu,x) \subset \{\R\} \cup \{T_r \S: r>0\}.
	\end{equation}
	Indeed, if not, then there exists a set $S\subset X$ of positive $\mu$ measure and, for each $x\in S$, an $r_x>0$ such that $T_{r_x} \T \in \Tan(\mu,x)$.
	Now, $\mu$-a.e. $x\in S$ satisfies the conclusion of \Cref{tan-of-tan} and \Cref{tan_is_R}.
	The former implies that $\R \in \Tan(\mu,x)$ and so the latter implies that $\Tan(\mu,x) =\{\R\}$, a contradiction.
	Equation \eqref{final_eq} and \Cref{tan_is_R} imply that we may decompose $X=E\cup S$ such that for $\mu$-a.e.\ $x\in E$, $\Tan(\mu,x)=\{\R\}$ and, for $\mu$-a.e.\ $x\in S$, $\Tan(\mu,x)=\{\lambda\mathcal S: \lambda>0\}$.

	Finally, \Cref{R_tan_rect} implies that $E$ is 1-rectifiable, whereas \Cref{tan_pu} and \Cref{tangent-density-equal} imply that $S$ is purely 1-unrectifiable.
\end{proof}

\section{Some solutions to the 1-dimensional Besicovitch problem}\label{sec:BP}
We now prove \Cref{solve_BP} by showing that $\S$ does not isometrically embed into any $X$ in the statement.
In fact we will show that, for sufficiently large $n$, the following metric space does not isometrically embed.
For $n\in\N$ let
\[
\S_n:=\{s_0,s_1,s_2,\ldots,s_n\}
\]
with the metric $\rho(0,i)=2^i$ and $\rho(i,j)=\rho(i,0)+\rho(0,j)$ for each $1\leq i\neq j\leq n$.

Once we have shown this, a compactness argument then implies that $\S_n$ does not $(1+\epsilon)$-bi-Lipschitz embed, for some $\epsilon>0$ (alternatively, this follows by adding quantitative details to each argument below).
Consequently, the set $S$ in \Cref{regular_tan} must have measure zero, completing the proof.

\subsection{Doubling geodesic metric spaces}
For $N\in\N$, a metric space $(X,d)$ is \emph{$N$-doubling} if any ball is covered by $N$ balls of half the radius.
A metric space is \emph{geodesic} if for any two points $x,y\in X$ there exists an arc-length parametrised $\gamma\colon [0,d(x,y)]\to X$ such that $\H(\gamma)=d(x,y)$.
For such a $\gamma$, $|s-t|=d(\gamma(s),\gamma(t))$ for each $s,t\in [0,d(x,y)]$.

Suppose that $X$ is geodesic and $N$-doubling and that, for some $n\in\N$, $\S_{n}$ isometrically embeds into $X$.
For each $1\leq i \leq n$ let $\gamma_i$ be a geodesic joining $s_0$ to $s_i$.
Then for any $1\leq i\neq j\leq n$ and $0\leq t_i \leq d(s_0,s_i)$,
\[d(\gamma_i(t_i),\gamma_j(t_j))= 2^{t_i}+2^{t_j}.\] 
In particular, $\gamma_i(2)\in B(s_0,2)$ and
\(d(\gamma_i(2),\gamma_j(2))=4\)
for each $1\leq i \neq j \leq n$.
Therefore, $B(s_0,2)$ is not contained within $n$ balls of radius $1$ and hence $n\leq N$.

\subsection{Uniformly convex Banach spaces}
A Banach space $(X,\|\cdot\|)$ is \emph{strictly convex} if, for every $x,y\in X$, $\|x+y\|=\|x\|+\|y\|$ implies that $x$ and $y$ are co-linear.
Observe that $S_3$ does not isometrically embed into a strictly convex Banach space.
Indeed, after translating so that $s_0=0$, such an embedding would require all $s_i$ to be co-linear, which violates the definition of $\rho$.

Further, $X$ is \emph{uniformly convex} if for every $0<\epsilon \leq 2$ there exists $\delta(\epsilon)>0$ such that $\|x-y\|\geq \epsilon$ implies $\|x+y\|\leq 2(1-\delta)$.
Uniform convexity is a quantitative notion of strict convexity.
Given a \emph{modulus of convexity} $\delta(\epsilon)$, the fact that $S_4$ does not isometrically embed into any (4-dimensional) strictly convex Banach space, there exists some $\eta>0$ such that $S_4$ does not $(1+\eta)$-bi-Lipschitz embed into a uniformly convex Banach space with modulus of convexity $\delta(\epsilon)$.

\subsection{Doubling metric spaces with dilations}

A metric space $(X,d)$ is equipped with \emph{dilations} if, for every $c\in X$ and $\lambda>0$, there exists a map $\delta_\lambda\colon X\to X$ with $f(c)=c$ such that
\[d(\delta_\lambda(x),\delta_\lambda(y))= \lambda d(x,y)\]
for all $x,y\in X$
and $\delta_{\lambda^-1}=\delta_{\lambda}^{-1}$.
Suppose also that there exist $0<\eta<1$ and $M\in\N$ such that, for any $x\in X$ and $m,n\geq M$,
\begin{equation}\label{growth}d(\delta_{2^m} x,\delta_{2^n} x) \leq \eta(2^m+2^n)d(c,x).\end{equation}

Now suppose that $X$ is $N$-doubling, let $l\in\N$ satisfy $0<2^{-l}<1-\eta$ and set $n=M+N^{l+1}$.
Suppose that $\S_n$ isometrically embeds into $X$ and let $\delta_\lambda$ be the dilations centred on $s_0$.
Cover $B(s_0,1)$ by balls $B_1,\ldots,B_{N^{l+1}}$ of diameter at most $2^{-l}$.
Then there exists some $B_i$ that contains two points $\delta_{2^{-m}}(s_m)$ and $\delta_{2^{-m}}(s_m)$ with $m,n\geq M$.
By the triangle inequality and the properties of dilations, in particular \eqref{growth},
\begin{align*}
       d(s_m,s_n)
       &=
       d(\delta_{2^m}(\delta_{2^{-m}}(s_m)),\delta_{2^n}(\delta_{2^{-n}}(s_n)))
       \\&\leq
       d(\delta_{2^m}(\delta_{2^{-m}}(s_m)),\delta_{2^m}(\delta_{2^{-n}}(s_n)))
       \\&\quad+
       d(\delta_{2^m}(\delta_{2^{-n}}(s_n)),\delta_{2^n}(\delta_{2^{-n}}(s_n)))
       \\&\leq 2^m 2^{-l} + \eta(2^m+2^n) \leq (\eta+2^{-l})(2^m+2^n).
\end{align*}
Since $\eta+2^{-l}<1$, this is impossible.

An example of a space with dilations satisfying \eqref{growth} is the \emph{Heisenberg group} $\mathbb{H}$ \cite[Definition 14.2.8]{zbMATH06397370}.
This is the set $\R^3$ equipped with the group law
\[(x,y,t)\star(x',y',t')=(x+x',y+y',z+z'+2(xy'-x'y)).\]
Note that 0 is the identity and $(x,y,t)^{-1}=(-x,-y,-t)$.
For $\lambda>0$, define
\[\delta_\lambda(x,y,t)=(\lambda x,\lambda y,\lambda^2 z).\]
A \emph{homogeneous norm} on $\mathbb{H}$ is a function $\|\cdot\|\colon \mathbb{H}\to\R$ such that $\|a^{-1}\|=\|a\|$,
\[\|a^{-1}\star b\| \leq \|a\|+\|b\|\]
and
\[\|\delta_\lambda(a)^{-1}\star \delta_\lambda(b)\|=\lambda\|a\star b^{-1}\|\]
for all $a,b\in \mathbb{H}$ and $\lambda>0$.
That is, $d(a,b)=\|a^{-1}\star b\|$ defines a metric with respect to which each $\delta_\lambda$ is a dilation centred on 0.
Since this metric is invariant under the group action, $\mathbb{H}$ has dilations centred on all points.
To see \eqref{growth}, let $(x,y,t)\in \mathbb{H}$ and $m,n\in\N$, and set $\lambda=2^m$ and $\mu=2^n$.
Then
\begin{align}
        \|\delta_\lambda(a)^{-1}\delta_\mu(a)\|
        &=
        \|(\mu-\lambda)x,(\mu-\lambda)y,(\mu^2-\lambda^2)t\|
        \notag
        \\&=
        |\mu-\lambda|\|x,y,\tfrac{\mu+\lambda}{\mu-\lambda}t\|
        \label{hberg}
        .
\end{align}
The topology with respect to $\|\cdot\|$ agrees with the Euclidean topology.
Since $|\tfrac{\mu+\lambda}{\mu-\lambda}|\leq 4$, there exists $C>0$ such that the final expression in \eqref{hberg} is at most $C|\mu-\lambda|$ for all $(x,y,t)\in B(0,1)$.

	\printbibliography
\end{document}